\documentclass{amsart}
\usepackage{amssymb,latexsym}
\usepackage{amsmath}
\usepackage[dvipdfm]{graphicx}
 \usepackage{color}
\usepackage{enumerate}
\newenvironment{enumeratei}{\begin{enumerate}[\upshape (i)]}{\end{enumerate}}
\numberwithin{equation}{section}
\theoremstyle{plain}
 \newtheorem{theorem}{Theorem}[section]
 \newtheorem{lemma}[theorem]{Lemma}
 \newtheorem{proposition}[theorem]{Proposition}
 \newtheorem{corollary}[theorem]{Corollary}
\theoremstyle{definition}
 \newtheorem{definition}[theorem]{Definition}
 \newtheorem{remark}[theorem]{Remark}

\newcommand \degx [2] {\textup{deg}_{#1}(#2)}
\newcommand \notmid {\mathrel{
\not\mathord{\kern -0.5pt \vert}}}
\newcommand \snotmid {\mathrel{
\not\mathord{\kern 2.2pt \vert}}}
\newcommand \binomial [2] {{{#1}\choose {#2}}}
\newcommand \vonal {\noalign{\hrule}}
\newcommand \modu[1] {\,\,\,(\text{mod }#1)}

\newcommand \url [1] {{\tt{#1}}}

\newcommand \tuple [1] {\langle #1\rangle}
\newcommand \pair [2] {\tuple{#1,#2}}
\renewcommand\phi{\varphi}
\newcommand \wt[1] {\widetilde{#1}}
\newcommand \wh[1] {\widehat{#1}}
\newcommand \qclo[1] {{#1}{}^{\scriptscriptstyle{{\pmb\square}}}}
\renewcommand\emptyset{\varnothing}

\newcommand \tbf [1] {\textbf{#1}} 
\newcommand \set[1] {\{#1\}}

\renewcommand\phi{\varphi}
\renewcommand\epsilon{{\varepsilon}}

\newcommand \nonparallel {\mathrel{
\not\mathord{\kern -1.5 pt\parallel}}}
\newcommand \id {\textup{id}}

\newcommand\init [1] {} 
\newcommand\nothing [1] {}
\newcommand \restrict[2] {{#1\rceil_{#2}}}
\newcommand \dom {\textup{Dom}}
\newcommand \setncl [1] {\textup {NCL}(#1)}
\newcommand \setncd [1] {\textup {NCD}(#1)}

\newcommand \fun[1]{{^\star\kern-2pt{#1}} }
\newcommand \ffun[1]{{^\star\kern-1pt{#1}} } 
\newcommand \varfun[1]{{^\star\kern-1pt{(#1)}} }
\newcommand \bvarfun[1]{{^\star\kern-1pt{\bigl(#1\bigr)}} }
\newcommand \idmap[1] {\textup{id}_{#1}}
\newcommand\algc [1] {{#1}^{{\textup{acl}}}}
\newcommand \ezg {g}
\newcommand \bix {\langle x\rangle}

\newcommand \pnct[1] {#1^{\pmb\circ}}
\newcommand \jelz[2] {#1^{(#2)}}
\newcommand \sinsinpol [1]{f^{\textup{s-s}}_{#1}}
\newcommand \sincospol [1]{f^{\textup{s-c}}_{#1}}
\newcommand \Woddodd [2]{W^{\textup{11}}_{#1,#2}}
\newcommand \Weveneven [2]{W^{\textup{00}}_{#1,#2}}
\newcommand \Wevenodd [2]{W^{\textup{01}}_{#1,#2}}
\newcommand \Woddeven [2]{W^{\textup{10}}_{#1,#2}}
\newcommand \Wallall [2]{W^{i_{#1}i_{#2}}_{#1,#2}}
\newcommand \Wnone [2]{W_{#1,#2}}
\newcommand \Ker {\textup{Ker}}

\long\def\nothing #1  {}
\long\def\needsoon #1  {}
\begin{document}
\title[Constructibility of cyclic polygons]
{Geometric constructibility of cyclic polygons and a limit theorem}

\author[G.\ Cz\'edli]{G\'abor Cz\'edli}
\email{czedli@math.u-szeged.hu}
\urladdr{http://www.math.u-szeged.hu/\textasciitilde{}czedli/}
\address{University of Szeged, Bolyai Institute. 
Szeged, Aradi v\'ertan\'uk tere 1, HUNGARY 6720}

\author[\'A.\ Kunos]{\'Ad\'am Kunos}
\email{akunos@math.u-szeged.hu}
\urladdr{http://www.math.u-szeged.hu/\textasciitilde{}akunos/}
\address{University of Szeged, Bolyai Institute. 
Szeged, Aradi v\'ertan\'uk tere 1, HUNGARY 6720}

\thanks{This research was supported by
the European Union and co-funded by the European Social Fund  under the project ``Telemedicine-focused research activities on the field of Mathematics, 
Informatics and Medical sciences'' of project number    ``T\'AMOP-4.2.2.A-11/1/KONV-2012-0073'', and by  NFSR of Hungary (OTKA), grant number 
K83219}

\dedicatory{To the eightieth birthday of Professor L\'aszl\'o Leindler}


\date{July 23, 2013; revised February 8, 2015\\
\phantom{nk} 2000 \emph{Mathematics Subject Classification}. Primary 51M04, secondary 12D05}

\begin{abstract} 
We study \emph{convex cyclic polygons}, that is, inscribed $n$-gons. Starting from P.\ Schreiber's idea, published in 1993, we prove that these polygons are not constructible from their \emph{side lengths} with  straightedge and compass, provided $n$ is at least five.  They are  non-constructible even in the particular case where they only have \emph{two} different \emph{integer} side lengths, provided that $n\neq 6$. To achieve this goal, we develop two tools of separate interest. First, 
we prove a \emph{limit theorem} stating  that, under reasonable  conditions, geometric constructibility is preserved under taking limits. To do so, we tailor a particular case of Puiseux's classical theorem on some generalized power series, called \emph{Puiseux series}, over algebraically closed
fields to an analogous  theorem on these series over real square root closed fields. Second,  based on \emph{Hilbert's irreducibility theorem},   we give a \emph{rational parameter theorem} that, under reasonable  conditions again, turns a non-constructibility result with a transcendental parameter into a non-constructibility result with a rational parameter. 
For $n$ even and at least six, we give an elementary proof for the non-constructibility of the cyclic $n$-gon from its side lengths and, also, from  the  \emph{distances} of its sides from the center of the circumscribed circle.  The fact that the cyclic $n$-gon is constructible from these distances for $n=4$ but non-constructible for  $n=3$  exemplifies that some conditions of the limit theorem cannot be omitted. 
\end{abstract}
\maketitle

\section{Introduction}\label{introsection}
\subsection{Target and some of the results}
A \emph{cyclic polygon} is a convex $n$-gon inscribed in a circle. Here $n$ denotes the \emph{order}, that is  the  number of vertices, of the polygon. \emph{Constructibility} is always understood as the classical geometric constructibility with  straightedge and  compass. Following Dummit and Foote~\cite[bottom of page 534]{dummitfoote}, we speak of an (unruled) \emph{straightedge} rather than a \emph{ruler}, because a ruler can have marks on it that we do not allow.  We know from Schreiber~\cite[proof of Theorem 2]{schreiber} that
\begin{equation}
\parbox{9.1cm}
{There exist positive real numbers $a,b,c$ such that   the cyclic pentagon with side lengths $a,a,b,b,c$ exists but it is  not  constructible from $a,b,c$ with straightedge and compass.}
\label{shrO5}
\end{equation}
Oddly enough, the  starting point of our research was that we could not understand the proof of Schreiber's next statement, \cite[Theorem 3]{schreiber}, which says that 
\begin{equation}
\parbox{9.1cm}
{If $n>5$, then the cyclic $n$-gon is in general not constructible from its side lengths with straightedge and compass.}
\label{schrng5}
\end{equation}
Note that \cite{schreiber} does not define
the meaning of ``in general not constructible''; we  analyze this concept later in the paper.
Supported by the details  given in the present paper later, we  think that the proof  of \eqref{schrng5} given in \cite{schreiber}  is incomplete. 
Fortunately, one can complete it with the help of  our limit theorem, Theorem~\ref{limthm}, which is of separate interest. Furthermore, the limit theorem leads to a slightly stronger statement.

We are only interested in the constructibility of a \emph{point} depending on  finitely many given points, because the constructibility of many other geometric objects, including cyclic polygons, reduces to this case easily.
A \emph{constructibility program} is a finite list of instructions that concretely prescribe which elementary Euclidean step for which points should be performed to obtain the next point. For example, 
\begin{equation}
\parbox{10.9 cm}{
``Take the intersection of the line through the ninth  and the thirteenth points with the circle whose center and radius are the first point and the distance between the fourth and sixth points, respectively."}
\label{constinSrt}
\end{equation}
 can be such an instruction. This  instruction is not always meaningful (e.g., the ninth and the thirteenth points may coincide and then they do not determine a line) and it can allow  choices (which intersection point should we choose). If there is a ``good'' choice at each instruction such that the last  instruction  produces the point that we intend to construct, then the constructibility program \emph{works} for the given data, that is, for the initial points. 
In our statements below, unless concrete data are mentioned,  
\begin{equation}
\parbox{9.8cm}{
A positive statement of \emph{constructibility}    means the existence of a  constructibility program that works \emph{for  all meaningful data} that define a non-degenerate cyclic polygon of the given order.}
\label{pzCrVy}
\end{equation} 

The cyclic polygon with side lengths $a_1,\dots, a_n$, in this order, is denoted by $P_n(a_1,\dots, a_n)$. 
 As usual, $\mathbb N:=\set{1,2,3,\dots}$ and $\mathbb N_0:=\set0\cup\mathbb N$.  For $i \in\mathbb N$, we define 
\begin{equation}
\parbox{8.2cm}
{\hskip -0.6cm $\setncl i =\bigl\{ n\in\mathbb N:$   $\exists \tuple{a_1,\dots, a_n}\in \mathbb N^n$ such that $P_n(a_1,\dots, a_n)$ exists, it is not constructible from $a_1,\dots, a_n$, and 
 $|\set{a_1,\dots, a_n}|\leq i\bigr\}$; }
\label{ncldef}
\end{equation}
the acronym comes from ``Non-Constructible from side Lengths''. Note that the
Gauss--Wantzel theorem, see  Wantzel~\cite{wantzel}, can be formulated in terms of $\setncl 1$; for example, $7\in\setncl 1$ and $17\notin\setncl 1$. More precisely, 
$n\in \setncl 1$ iff the regular cyclic $n$-gon is non-constructible iff $n$ is not of the form $2^k p_1\cdots p_t$ where $k,t\in\mathbb N_0$ and $p_1,\dots,p_t$ are pairwise distinct Fermat primes.
If $n$ belongs to  $\setncl i$ for some $i \in\mathbb N$, then the  the cyclic $n$-gon is not constructible in  our ``concrete'' sense given above or in any reasonable ``general'' sense. Clearly, for $i=1$,    $\mathbb N^n$  in   \eqref{ncldef}  can be replaced instead of $\mathbb R^n$, because the unit distance at a geometric construction is up to our choice.
However, for $i>1$,  $\mathbb N^n$  in   \eqref{ncldef}  rather than $\mathbb R^n$  makes the  result below stronger. 
One of our  goals is to prove the following theorem, which is a stronger statement than \eqref{schrng5}. 
Parts \eqref{thmourc} and \eqref{thmourd} below can be combined; however, we  formulate them separately, because  we have an elementary proof for  \eqref{thmourc} but not for \eqref{thmourd}. Part \eqref{thmourbc} is well-known.

\goodbreak
\begin{theorem}\label{thmour}\ 
\begin{enumeratei}
\item\label{thmoura} For $n\in\set{3,4}$,  the cyclic $n$-gon is constructible in general from its side lengths with straightedge and compass.
\item\label{thmourb} 
\begin{enumerate}[\upshape (a)]
\item $5\in\setncl 2\setminus\setncl 1$.
\item\label{thmourbb}  $6\in\setncl 3$ but $6\notin\setncl 2$. Furthermore, if $a_1,\dots,a_6$ are positive \emph{real}  numbers such that $P_6(a_1,\dots,a_6)$ exists and $| \set{a_1,\dots, a_6}|\leq 2$, then  the cyclic hexagon $P_6(a_1,\dots,a_6)$ can be constructed from its side lengths.  
\item\label{thmourbc}  $7\in\setncl 1$.
\end{enumerate}
\item\label{thmourc}
If $n\geq 8$ is an \emph{even} integer, then $n\in\setncl 2$.
\item\label{thmourd}
If $n\geq 8$ is an \emph{odd} integer, then $n\in\setncl 2$.
\end{enumeratei}
\end{theorem}
Note that if  the regular $n$-gon is non-constructible, then $n\in\setncl 1$. However, if the regular $n$-gon is constructible, 
and infinitely many $n\geq 8$ are such,
then $n\notin\setncl 1$ and, for these $n$, parts \eqref{thmourc} and \eqref{thmourd} above cannot be strengthened by changing  $\setncl 2$ to $\setncl 1$. 
The elementary method we use to prove Part \eqref{thmourc} of Theorem~\ref{thmour} easily leads us to the following  statement on cyclic $n$-gons of \emph{even order}; see Figure~\ref{fig1} for an illustration. Assume that, with straightedge and compass,  we want to construct the cyclic $n$-gon $D_n(d_1,\dots, d_n)$ from the distances $d_1,\dots,d_n$ of its sides from the center of its circumscribed circle. We define $\setncd i$ analogously to \eqref{ncldef}; now the acronym comes from `Non-Constructible from Distances''.

\begin{theorem}\label{thmdi}  If $n\geq 6$ is even, then  $n\in \setncd 2$.
\end{theorem}

Evidently, $n\in\setncd 1$ iff the regular $n$-gon is non-constructible.
To shed more light on Theorem~\ref{thmdi}, we recall the following statement from  Cz\'edli and Szendrei~\cite[IX.1.26--27,2.13 and page 309]{czgsza}, which was proved by computer algebra; note that we do not claim that $A_i\cap \setncd{i-1}=\emptyset$ holds in it.

\begin{proposition}[\cite{czgsza}] \label{prop100}
Let $A_4:=\set{6,8}$, $A_3:=\set{3, 5, 12,24,30}$, 
\[
A_2:=\set{10, 15, 16, 17, 20, 32, 34, 40, 48, 51, 60, 64, 68, 80, 85, 96},
\]
and $A_1:=\set{3,5,6,7,\dots,100}\setminus(A_2\cup A_3 \cup A_4)$.  Then, for every $i\in\set{1,2,3,4}$, 
$A_i\subseteq \setncd i$. As opposed to $D_3(d_1,\dots,d_3)$, $D_4(d_1,\dots,d_4)$ is constructible  from  $\tuple{d_1,\dots,d_4}$.
\end{proposition}

The following statement, which we recall from Cz\'edli and Szendrei~\cite[IX.2.14]{czgsza}, extends the scope of Theorem~\ref{thmdi}  to circumscribed polygons.

\begin{remark}[\cite{czgsza}] Let $3\leq n\in\mathbb N$.
With the notation given before Theorem~\ref{thmdi}, 
a circumscribed  $n$-gon $T_n$ is constructible from the distances $t_1,\dots,t_n$ of its vertices from the center of the inscribed circle if and only if  the inscribed polygon $D_n(1/t_1,\dots,1/t_n)$ is constructible from  $\tuple{1/t_1,\dots,1/t_n}$ or, equivalently, from  $\tuple{t_1,\dots,t_n}$.
\end{remark}

\begin{figure}[htb]
\centerline
{\includegraphics[scale=1.0]{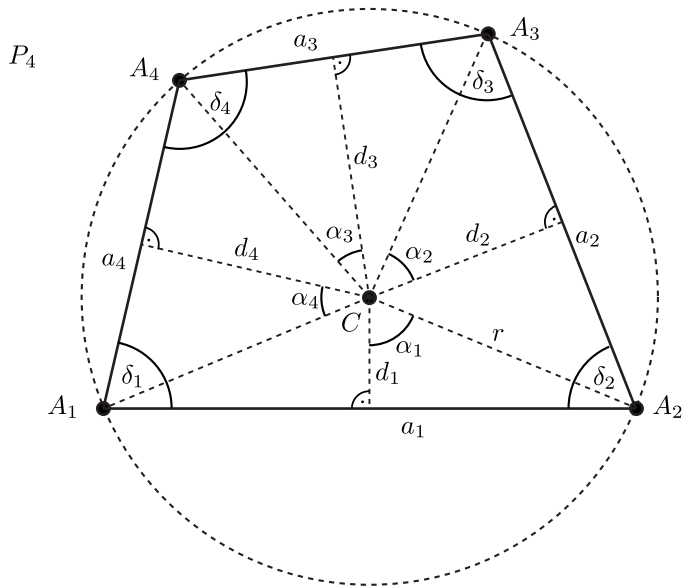}}
\caption{A cyclic $n$-gon for $n=4$\label{fig1}}
\end{figure}

\subsection{Prerequisites and outline} 
Undergraduate, or sometimes graduate, mathematics is sufficient to follow the paper. The reader is assumed to know the rudiments of simple field extensions and that of calculus. Following, say, Cohn~\cite[page 9]{cohn}, Gr\"atzer~\cite[page 1]{gratzerUA}, and R\'edei~\cite[page 12]{redei}, the notation $X\subset Y$ stands for proper inclusion, that is, $X\subset Y$ iff $X\subseteq Y$ and $X\neq Y$.
 
The paper is structured as follows. Section~\ref{ntsect} gives Schreiber's argument for \eqref{schrng5}; the reader can form his or her own evaluation before reading the present paper further. 
Section~\ref{evensection} gives an elementary proof for part~\eqref{thmourc} of  
Theorem~\ref{thmour}, that is, for all even $n\geq 8$; this section also proves Theorem~\ref{thmdi} for $n\geq 8$.
Section~\ref{smallorder} is devoted to cyclic polygons of small order, that is, for $n<8$; here we prove parts~\eqref{thmoura}--\eqref{thmourb} of  Theorem~\ref{thmour} and  the case $n=6$ of  Theorem~\ref{thmdi}. Also, this section  recalls some arguments from  \cite{czgsza} to prove some parts of  Proposition~\ref{prop100}.
In Section~\ref{moreabout}, we comment on Schreiber's argument. Section~\ref{basefieldext} collects some basic facts on field extensions. In particular, this section gives a rigorous algebraic treatment for  real functions composed from the four arithmetic operations and $\sqrt{\phantom o}$.  Section~\ref{powsersect} proves that the functions from the preceding section can be expanded into \emph{power series with dyadic rational exponents} such that the coefficients of these series are geometrically constructible. Section~\ref{sectpuiseuxcomments}
compares  these  expansions 
 with  Puiseux series and Puiseux's theorem.
Based on our expansions from Section~\ref{powsersect},    we prove a \emph{limit theorem} for geometric constructibility in Section~\ref{limthemsect}. Using this theorem, 
Section~\ref{limatworksection} proves a weaker form of parts~\eqref{thmourc}--\eqref{thmourd} of Theorem~\ref{thmour}, with transcendental parameters rather then integer ones, and points out how one could complete Schreiber's argument. Armed with Hilbert's irreducibility theorem, Section~\ref{hilbertsection} proves a  
\emph{rational parameter theorem} that, under reasonable  conditions, turns a non-constructibility result with a transcendental parameter into a non-constructibility result with a rational parameter. Finally, based on the tools elaborated in the earlier sections, Section~\ref{compLsect} completes the proof of Theorem~\ref{thmour} in few lines.

Since the Limit Theorem and the Rational Parameter Theorem are about geometric constructibility in general, not only for cyclic polygons, they can be of separate interest.

\section{Schreiber's argument}\label{ntsect}
Most of Schreiber~\cite{schreiber} is clear and
practically all mathematicians can follow it.  
We only deal with  \cite[page 199, lines 3--15]{schreiber}, where, in order to prove \eqref{schrng5}, 
  he claims to perform the induction step from $(n-1)$-gons to $n$-gons.
His argument  is basically the following paragraph; the only difference is that we use the radius (of the circumscribed circle) rather than the coordinates of the vertices. This simplification is not an essential change, because 
the (geometric) constructibility of a cyclic $n$-gon is clearly equivalent to the  constructibility  of its radius.  

Suppose, for a contradiction, that the radius of the cyclic $n$-gon is in general constructible from the  side lengths $a_1,\dots,a_n$. Hence, this radius is a  quadratic irrationality $R$ depending on the variables $a_1,\dots,a_n$, and such as it is a {continuous function} of its  $n$ variables. On the other hand, the geometric dependence of the radius from  $a_1,\dots,a_n$  is described by a continuous function $f$ of the same variables. Because for $a_n\to 0$ the radius of the $n$-sided inscribed polygon converges to that of the  $(n-1)$-sided polygon with side lengths $a_1,\dots, a_{n-1}$ and the continuous functions $R$ and $f$ are identical for $a_n\neq  0$, $R$  takes the same limit value for $a_n\to 0$ as $f$. That is, for $a_n=0$, the quadratic irrationality $R$ describes the constructibility of the radius of the inscribed $(n-1)$-gon. Finally, iterating the same process, we obtain that the radius of the cyclic $(n-2)$-gon, that of the cyclic $(n-3)$-gon, \dots, that of the cyclic $5$-gon are constructible, which contradicts \eqref{shrO5}.

Before we analyze Schreiber's argument in Section~\ref{moreabout}, the reader is invited to form his or here own opinion. Note that  \cite{schreiber} does not define what ``quadratic irrationalities'' are.

\section{An elementary proof for $n$ even}
\label{evensection}
Let $a_1$,\dots, $a_n$ be arbitrary positive real numbers. 
It is proved in Schreiber~\cite[Theorem 1]{schreiber} that 
\begin{equation}
\parbox{8.5cm}
{There exists a cyclic $n$-gon with side lengths $a_1,\dots, a_n$ iff $a_j<\sum\set{a_i: i\neq j}$ holds for every $i\in\set{1,\dots, n}$.
}\label{shcrnqltS}
\end{equation} 
Our elementary approach will be based on the following well-known statement from classical algebra.  Unfortunately, a thorough treatment of geometric constructibility is usually missing from current books on algebra in English, at least  in our reach; so it is not so easy to give references. 
Part~\eqref{szrka} below is 
Herstein~\cite[Theorem 5.5.2 in page 206]{herstein} and \cite[Theorem III.3.1 in page 63]{czgsza}. 
Part~\eqref{szrkb} is the well-known  Eisenstein-Sch\"onemann criterion, see Cox~\cite{cox} for our terminology.  
Part~\eqref{szrkc} is less elementary and will only be used in Section~\ref{hilbertsection}, but even this part is often taught for graduate students. This part is \cite[Theorem V.3.6]{czgsza}, and also Kiss~\cite[Theorem 6.8.17]{kiss},   Kersten~\cite[Satz in page 158]{kersten}, and Jacobson~\cite[Criterion 4.11.B in page 263]{jacobsonI}. It also follows from Gilbert and Nicholson~\cite[Theorem 13.5 in page 254]{gilbertnicholson} (combined with Galois theory).  
The degree of a polynomial $f=f(x)$ is denoted by $\deg(f)$, or by $\degx xf$ if we want to indicate the variable. Let 
$a_1,\dots, a_k$, and $b$ be real numbers; as usual, the smallest subfield $K$ of $\mathbb R$ such that $\set{a_1,\dots,a_k}\subseteq K$ is denoted by $\mathbb Q(a_1,\dots,a_n)$. In this case, instead of ``$b$ is constructible from $a_1,\dots,a_k$'', we can also say that $b$ is \emph{constructible over} the field $K$. We shall use this terminology only for finitely generated subfields of $\mathbb R$. 
By definition, a \emph{complex number is constructible} if both of its real part and imaginary part are constructible. Equivalently, if it is constructible as a point of the plane. 

\goodbreak

\begin{proposition}\label{szrk} \ 
\begin{enumerate}[\upshape (A)]
\item\label{szrka} If $f \in\mathbb Q[x]$ is an irreducible polynomial in $\mathbb Q[x]$, $c\in\mathbb R$,  $f(c)=0$, and the degree $\deg(f)$ is not a power of $\,2$, then $c$ is not constructible over $\mathbb Q$. 
\item\label{szrkb} 
{If $f(x)=\sum_{j=0}^k a_jx^j\in \mathbb Z[x]$ and $p$ is a prime number such that $p\notmid a_k$, 
$p^2\notmid a_0$, and $p\mid a_j$ for $j\in\set{0,\dots, k-1}$, then $f(x)$ is irreducible in $\mathbb Q[x]$.}
\item\label{szrkc} Let $K$ be a finitely generated subfield of $\mathbb R$. 
If $f \in K[x]$ is an irreducible polynomial in $K[x]$, $c\in\mathbb R$,  and $f(c)=0$, then  $c$ is  constructible over $K$ if and only if the degree of the splitting field of $f$ over $K$ is a power of $\,2$.
\end{enumerate}
\end{proposition}

For $k\in  \mathbb N$, we need the following two known formulas, which are easily derived from de Moivre's formula and the binomial theorem. For brevity, the conjunction of ``$2\mid j$'' and ``$j$ runs from $0$'' is denoted by $2\mid j=0$, while $2\notmid j=1$ is understood analogously.
\begin{align}
\sin(k\gamma)&=\sum_{2\snotmid j=1}^k (-1)^{(j-1)/2} \binomial k j ( \cos \gamma)^{k-j}\cdot (\sin\gamma)^{j},\label{sinsum}\\
\cos(k\gamma)&=\sum_{2\mid j=0}^k (-1)^{j/2} \binomial k j ( \cos \gamma)^{k-j}\cdot (\sin\gamma)^{j}\text.\label{cossum} 
\end{align}

A prime $p$ is a \emph{Fermat prime}, if $p-1$ is a power of $2$. A Fermat prime is necessarily of the form $p_k=2^{2^k}+1$. We know that $p_0=3$, $p_1=5$, $p_2=17$, $p_3=257$, and $p_4=65\,537$ are Fermat primes, but it is an open problem if there exists another Fermat prime.

\begin{lemma}\label{nakuml} If $n=5$ or $8\leq n\in\mathbb N$, then there exists a prime $p$ such that $n/2<p<n$ and $p$ is not a Fermat prime.
\end{lemma}

\begin{proof} 
We know from  Nagura~\cite{nagura} that, for each $25\leq x\in \mathbb R$, there exists a prime in the open interval
$(x, 6x/5)$. Applying this result twice, we obtain two distinct primes in $(x,36x/25)$. Hence, for $25\leq n\in \mathbb N$, there are at least two primes in the interval $(n,2n)$. 
Since the ratio of two consecutive Fermat primes above 25 is more than 2, this gives the lemma for $50\leq n$. For $n\leq 50$, appropriate primes are given in  the  following table.
\[
\lower  0.0 cm
\vbox{\tabskip=0pt\offinterlineskip
\halign{\strut#&\vrule#\tabskip=5pt plus 2pt&
#\hfill& \vrule\vrule\vrule#&
\hfill#&\vrule#&
\hfill#&\vrule#&
\hfill#&\vrule#&
\hfill#&\vrule#&
\hfill#&\vrule\tabskip=0.0pt#&
#\hfill\vrule\vrule\cr
\vonal
&&$n$&&$5$&&$8$--${13}$&&${14}$--$25$&&$26$--${45}$&&$46$--${85}$&\cr
\vonal\vonal\vonal
&&$p$&&$3$&&$7$&&${13}$&&${23}$&&$43$&\cr
\vonal}}\qedhere  
\]
\end{proof}

\begin{proof}[Proof of Theorem~\ref{thmour} \eqref{thmourc}] Let $n\geq 8$ be even. It suffices to find an  appropriate $p$ in the set $\set{1,2,\dots,n-1}$ and  $a, b\in \mathbb N$  such that $P_n$ is not constructible even if $p$ of the given $n$ side lengths are equal to  $a$ and the rest $n-p$ side lengths are equal to $b$, for appropriate integers $a$ and $b$. Let $r$ and $C$ be the radius and the center of the  circumscribed circle, respectively. 

The half of the central angle for $a$ and $b$ are denoted by $\alpha$ and $\beta$, respectively; see the $\alpha_i$ in Figure~\ref{fig1} for the meaning of half central angles.
Clearly, $P_n$ is constructible if{f} so is $u=1/(2r)$. 
Since we will choose $a$ and $b$ nearly equal, $C$ is in the interior of $P_n$, and we have 
\begin{equation}\label{alppsn}
 p\alpha+(n-p)\beta =\pi\text.
\end{equation} 

It follows from \eqref{alppsn} that 
 $\sin(p\alpha)-\sin((n-p)\beta)=0$. Therefore, using \eqref{sinsum}, 
\begin{equation}
\text{$\sin\alpha=au$,  $\sin\beta=bu$, $\cos\alpha=\sqrt{1-a^2u^2}$, and $\cos\beta=\sqrt{1-b^2u^2}$,}
\label{sncnsncnd}
\end{equation}
we obtain that $u$ is a root of the following function:
\begin{equation}
\begin{aligned}
f_p^{(1)}(x)&=\sum_{2\snotmid j=1}^{p}  (-1)^{(j-1)/2} \binomial p j (1-a^2x^2)^{(p-j)/2}\cdot (a x)^{j}\cr
&- \sum_{2\snotmid j=1}^{n-p}  (-1)^{(j-1)/2} \binomial {n-p}j (1-b^2x^2)^{(n-p-j)/2}\cdot (b x)^{j}\cr
&=\Sigma^f_{1}-\Sigma^f_{2}\text.
\end{aligned}\label{fpsin} 
\end{equation}
Observe that $f_p^{(1)}(x)$ is a polynomial since $p-j$ and $n-p-j$ are even for $j$  odd. In fact, $f_p^{(1)}(x)\in\mathbb Z[x]$ for all $a,b\in\mathbb N$. 
Besides $f_p^{(1)}(x)=\Sigma^f_{1}-\Sigma^f_{2}$, we also consider the polynomial $f_p^{(2)}(x)=\Sigma^f_{1}+\Sigma^f_{2}$. 

From now on, we assume that 
\begin{equation}
\text{
$8\leq n$ is even and $p$ is chosen according to Lemma~\ref{nakuml}.}
\label{pcHoose}
\end{equation}
It is obvious by \eqref{shcrnqltS} that
we can  choose positive integers $a$ and $b$ such that 
\begin{equation}\label{abmodppv}
a\equiv 1 \modu{p^2},\quad b\equiv 0 \modu{p^2}, 
\end{equation}
and $a/b$ is so close to 1 that $P_n$ exists and $C$ is in the interior of $P_n$. The inner position of $C$ is convenient but not essential, because we can allow a central angle larger than $\pi$; then  \eqref{sncnsncnd} still holds and the sum of half central angles is still $\pi$.

Let $v\in\set{1,2}$.
The assumption $n/2<p<n$ gives $\degx x{f_p^{(v)}}=p$. 
Hence, we can write
\[ f_p^{(v)}(x)=\sum_{s=0}^p c^{(v)}_s x^s,\quad\text{where }\, c^{(v)}_0,\dots, c^{(v)}_p\in\mathbb Z\text.
\]
We have $c^{(v)}_0=0$ since $j>0$ in \eqref{fpsin}.  Our plan is to apply Proposition~\ref{szrk}\eqref{szrkb} to the polynomial $f_p^{(v)}(x)/x$.  Hence, we are only interested in the coefficients $c^{(v)}_s$ 
modulo $p^2$. Note that this congruence extends to the polynomial ring $\mathbb Z[x]$ in the usual way.
The presence of $(bx)^j$ in $\Sigma^f_{2}$ yields that all coefficients in $\Sigma^f_{2}$ are congruent to $0$ modulo  $p^2$.  Therefore, $f_p^{(v)}(x)\equiv \Sigma^f_{1}\modu{p^2}$, and we can assume that the all the $c^{(v)}_s$ come from $\Sigma^f_{1}$. Each summand of $\Sigma^f_{1}$ is of degree $p$. Therefore, computing modulo $p^2$, the leading coefficient $c^{(v)}_p$  satisfies the following:
\begin{equation}\label{cepe}
\begin{aligned}
c^{(v)}_p &\equiv \sum_{2\snotmid j=1}^{p}  (-1)^{(j-1)/2} \binomial p j
(-1)^{{(p-j)/2}}\, (a^2)^{(p-j)/2}\, a^{j} \cr
& = {(-1)^{(p-1)/2}} \sum_{2\snotmid j=1}^{p} \binomial p j a^{p}\equiv (-1)^{(p-1)/2}\sum_{2\snotmid j=1}^{p} \binomial p j \cr
&= {(-1)^{(p-1)/2}}\,2^{p-1} = {(-1)^{(p-1)/2}}+p t_p\modu{p^2 }\quad\text{for some }t_p\in\mathbb Z;
\end{aligned}
\end{equation}
the last but one equality is well-known while the last one follows from  Fermat's  little theorem. Since $\Sigma^f_{1}$ gives a linear summand only for $j=1$, we have 
\begin{equation}\label{ceegy}
c^{(v)}_1\equiv \binomial p1 \cdot a=pa \equiv p\modu{p^2}\text.
\end{equation}
Next, let $1\leq s<p$. For $j=p$, the $j$-th summand of $\Sigma^f_{1}$ is $\pm (ax)^p$, which cannot influence $c^{(v)}_s$. Hence, modulo $p^2$, $c^{(v)}_s$ comes from the $\sum_{2\snotmid j=1}^{p-2}$ part of $\Sigma^f_{1}$.  However, for $j\in\set{1,\dots,p-2}$, the binomial coefficient $\binomial pj$ is divisible by $p$. Hence, we conclude that there exist integers $t_1,\dots, t_{p-1}$ such that 
\begin{equation}\label{cees}
c^{(v)}_s\equiv p t_s \modu{p^2} \quad\text{ for }s\in\set{1,\dots,p-1}\text.
\end{equation}
Now, \eqref{cepe}, \eqref{ceegy}, \eqref{cees}, $c_0^{(v)}=0$,  and Proposition~\ref{szrk}\eqref{szrkb} imply that 
\begin{equation}\label{fstarirr}
\text{for $v=1,2$,\quad $f_p^{(v)}(x)/x\,$ is irreducible.}
\end{equation} 
By the choice of $p$, $\degx x{f_p^{(v)}(x)/x}=p-1$ is not a power of 2. 
Since $a,b\in \mathbb Z$, we can apply Proposition~\ref{szrk}\eqref{szrka}  to $f_p^{(1)}(x)/x$  to conclude that $P_n$ is not constructible. This proves Theorem~\ref{thmour} \eqref{thmourc}.
\end{proof}

\begin{proof}[Proof of Theorem~\ref{thmdi} for $n\geq 8$] 
Let $p$ be a prime according to Lemma~\ref{nakuml}. 
Choose $a,b\in \mathbb N$  according to \eqref{abmodppv} such that   $a/b$ be sufficiently close to 1. Let  $d_1=\dots=d_p=a$ and $d_{p+1}=\dots= d_n=b$  be the distances of the sides of $D_n$ from $C$. Hence, $D_n=D_n(a,\dots,a,b,\dots,b)$ exists and, clearly, 
its interior contains the center $C$ of the circumscribed circle. 
(Note that the inner position of $C$ is convenient but not essential if we allow that one of the given distances can be negative.) 
The radius of  the circumscribed circle is denoted by $r$, and let $u=1/r$.  
Instead of \eqref{sncnsncnd}, now we have 
\begin{equation}\label{Tsncnsncnd}
\text{$\cos\alpha=au$,  $\cos\beta=bu$, $\sin\alpha=\sqrt{1-a^2u^2}$, and $\sin\beta=\sqrt{1-b^2u^2}$.}
\end{equation}
Combining \eqref{cossum},  \eqref{alppsn}, and \eqref{Tsncnsncnd}, and using $2\notmid p$ and  $2\notmid n-p$,  we obtain that $u$ is a root of the following polynomial:
\begin{equation}
\begin{aligned}
g_p(x)&=\sum_{2\mid j=0}^{p-1}  (-1)^{j/2} \binomial p j (a x)^{p-j} (1-a^2x^2)^{j/2} \cr
&+ \sum_{2\mid j=0}^{n-p-1}  (-1)^{j/2} \binomial {n-p}j (b x)^{n-p-j}(1-b^2x^2)^{j/2} =\Sigma^g_{1}+\Sigma^g_{2}\text.
\end{aligned}\label{fpcos} 
\end{equation}
Substituting $s$ for $p-j$ in $\Sigma^g_{1}$ above and using the rule $\binomial pj=\binomial p{p-j}$, we obtain $\Sigma^g_{1}=(-1)^{(p-1)/2}\cdot \Sigma^f_{1}$. Similarly, 
substituting $s$ for $n-p-j$ in $\Sigma^g_{2}$, we obtain $\Sigma^g_{2}=(-1)^{(n-p-1)/2}\cdot \Sigma^f_{2}$. Hence, $\set{g_p(x), -g_p(x)}\cap \set{f_p^{(1)}(x),f_p^{(2)}(x)}\neq\emptyset$, and \eqref{fstarirr} yields that $g_p(x)/x$ is irreducible. 
Therefore, Proposition~\ref{szrk}  implies that $D_n(a,\dots,a,b,\dots, b)$ is not constructible. This proves  Theorem~\ref{thmdi} for the case $2\mid n\geq 8$.
\end{proof}

\section{Cyclic polygons of small order}
\label{smallorder}
The ideas we use in this section are quite easy. However, the concrete computations often  require and almost always make it reasonable  to use computer algebra. The corresponding Maple worksheet (Maple version V.3 of November 27, 1997) is available from the authors' web sites.

\begin{definition} For $k,m\in\mathbb N$ and $a,b\in\mathbb R$, we define the following polynomials. (We will soon see that the mnemonic superscripts s and c come from sine and cosine; the first one refers to  a single angle while the second one to a multiple angle. The superscripts 0 and 1 refer to the parity of the subscripts.)
\allowdisplaybreaks{
\begin{align*}
&\begin{aligned}
\sinsinpol k(x) &:=\sum_{2\snotmid j=1}^k (-1)^{(j-1)/2} \binomial k j ( 1-x^2)^{(k-j)/2}\cdot x^{j},\quad\text{for }k\text{ odd},\cr
\sincospol k(x)  &: =\sum_{2\mid j=0}^k (-1)^{j/2} \binomial k j ( 1-x^2)^{(k-j)/2}\cdot x^{j},  \quad\text{for }k\text{ even,}\cr
\end{aligned} \cr
&\begin{aligned}
\Woddodd k m(a,b,x)&:= \sinsinpol k(a x)-\sinsinpol m(b x), \text{ for }k,m \text{ odd,}\cr
\Weveneven k m(a,b,x)&:= \sincospol k(a x)+\sincospol m(b x), \text{ for }k,m \text{ even,}\cr
\Woddeven k m(a,b,x)&:= \bigl(\sinsinpol k(a x)\bigr)^2+\bigl(\sincospol m(b x)\bigr)^2-1, \text{ for }k \text{ odd and } m \text{ even,}\cr
\Wevenodd k m(a,b,x)&:= \bigl(\sincospol k(a x)\bigr)^2+\bigl(\sinsinpol m(b x)\bigr)^2-1,  \text{ for }k \text{ even and } m \text{ odd}\cr
\Wnone k m(a,b,x)&:=\Wallall k m(a,b,x),\text{ where }i_k\equiv k \text{ and }i_m\equiv m\text{ (mod 2).}
\end{aligned}
\end{align*}%
}
\end{definition}

\begin{lemma}\label{LmMrZT}
Let $k,m\in\mathbb N$ and $0< a,b\in \mathbb R$.
\begin{enumeratei}
\item\label{LmMrZTa} If $k$ is odd, then $\sinsinpol k(x)\in\mathbb Z[x]$ is a polynomial of   degree $k$ and, for  $\alpha\in \mathbb R$,    $\sinsinpol k(\sin(\alpha))= \sin(k\alpha)$. The leading coefficient of $\sinsinpol k(x)$ is $(-1)^{(k-1)/2}\cdot2^{k-1}$.
\item\label{LmMrZTb} If $k$ is even, then $\sincospol k(x)\in\mathbb Z[x]$ is a polynomial of   degree $k$ and, for  $\alpha\in \mathbb R$,    $\sincospol k(\sin(\alpha))= \cos(k\alpha)$. The leading coefficient of $\sincospol k(x)$ is $(-1)^{k/2}\cdot2^{k-1}$.
\item\label{LmMrZTc} $\Wnone km(a,b,x)$ is a polynomial with indeterminate 
$x$. If the parameters $a$ and $b$ are also treated as indeterminates, then $\Wnone km(a,b,x)$ is a polynomial over $\mathbb Z$. 
For  $0<a\in \mathbb R$ and $0<b\in \mathbb R$, if $a \neq b$, or $k\neq m$,  or  $k=m$ is even, then  $\Wnone km(a,b,x)$  is not the zero polynomial. 
 Furthermore, if the cyclic polygon
\begin{equation}
P_{n}(\underbrace{a, \dots, a}_{\textup{$k $ copies}},\underbrace{b, \dots, b}_{\textup{$m$ copies}})
\label{scABvzPT}
\end{equation}
exists and  $r$ denotes the radius of its circumscribed circle, then we have that $\Wnone km(a,b,1/(2r))=0$. 
\end{enumeratei}
\end{lemma}

\begin{proof}  Since $0=(1-1)^k=\sum_{j=0}^k (-1)^j\binomial kj$ and  $2^k=\sum_{j=0}^k \binomial kj$,
\[ \sum_{2\snotmid j=1}^k  \binomial kj=2^{k-1}\quad\text{and}\quad \sum_{2\mid j=0}^k  \binomial kj=2^{k-1}\text.
\] 
We conclude easily from  these equalities, \eqref{sinsum}, and \eqref{cossum} that parts \eqref{LmMrZTa} and \eqref{LmMrZTb} of the lemma hold. These parts imply that   $\Wnone km(a,b,x)$ is a polynomial and, for positive real numbers $a$ and $b$, it  is the zero a polynomial iff $a=b$, $k=m$, and $k=m$  is odd. 
Let $u=1/(2r)$.
Denoting the half of the central angle for $a$ and $b$ by  $\alpha$ and $\beta$, as in the proof of Theorem~\ref{thmour} \eqref{thmourc}, 
we obtain that $\sin(\alpha)=au$ and $\sin(\beta)=bu$. Let $\wh\alpha=k\alpha$ and $\wh\beta=m\beta$. Since $\wh\alpha+\wh\beta=\pi$, we have
\begin{equation}
\begin{aligned}
&\sin(\wh\alpha)=\sin(\wh\beta),\quad \cos(\wh\alpha)=-\cos(\wh\beta), \cr
& \bigl(\sin(\wh\alpha) \bigr)^2 + \bigl(\cos(\wh\beta) \bigr)^2  = 1,
\quad  \bigl(\cos(\wh\alpha) \bigr)^2 +  \bigl(\sin(\wh\beta) \bigr)^2 = 1\text.
\end{aligned}
\label{sdkGrR}
\end{equation}
Therefore, part \eqref{LmMrZTc} follows easily  from 
parts \eqref{LmMrZTa} and \eqref{LmMrZTb} and \eqref{sdkGrR}.
\end{proof}

\begin{proof}[Proof of Theorem~\ref{thmdi} for $n=6$]
We follow Cz\'edli and  Szendrei~\cite[IX.2.13]{czgsza}; only the values of the $d_i$ are different. Let 
\begin{equation}\label{daTs}
\text{$d_1=d_2=d_3=d_4=1000$, $d_5=999$, and $d_6=1001$.}
\end{equation}
Using continuity, it is a straightforward but tedious to show that $D_6(d_1,\dots, d_6)$  exists; the details are omitted. Let $\alpha_1,\dots, \alpha_6$ denote the  central half angles.  As usual, $\cos(\alpha_5)=d_5u=999u$, where $u=1/r$, and $\cos(\alpha_6)=d_6u=1001u$. We obtain from \eqref{cossum}  and $\cos(\alpha_1)=d_1u=1000u$ that $\cos(\alpha_1+\dots+\alpha_4)=\cos(4\alpha_1) = 8(\cos(\alpha_1))^4  - 8(\cos(\alpha_1))^2 +1   = 8\cdot 10^{12}u -  8\cdot 10^{6}u+1$.
  These equalities,    \eqref{thrsmm}, which we recall  from \cite{czgsza} soon,  and $4\alpha_1+\alpha_5+\alpha_6=\pi$     imply that $u$, which is not $0$,  is a root of a polynomial of degree 8 in $\mathbb Z[x]$. We easily obtain this polynomial by computer algebra. 
It  is divisible by $4\,000\,000x^2$ and contains no summand of odd degree. Therefore, dividing the polynomial by $4\,000\,000x^2$, we obtain that $u^2$ is a root of 
\[16\cdot 10^{18}\cdot  x^3-28\,000\,004 \cdot 10^6 \cdot   x^2+16\,000\,004\cdot  x-3\text.
\]
By computer algebra, this polynomial is irreducible. Hence $u^2$ is not constructible, implying that none of  $u$, $r=1/u$, and  $D_6(d_1,\dots, d_6)$ is constructible.  
This completes the proof of Theorem~\ref{thmdi}.
\end{proof}

\begin{proof}[Proof of Theorem~\ref{thmour} \eqref{thmoura} and \eqref{thmourb}]  Unless otherwise stated, we keep the notation from  the  proof of part~\eqref{thmourc}.  In particular, $u=1/(2r)$. The case $n=3$ is trivial.

Assume $n=4$. With the notation of Figure~\ref{fig1} and using the fact that  $\cos\delta_3 =\cos(\pi-\delta_1)=-\cos \delta_1$, the law of cosines gives  
\begin{equation}
a_1^2+a_3^2-2a_1a_3\cos \delta_1=\overline{A_2A_4}^{\,2}= a_2^2+a_4^2+2a_2a_4\cos \delta_1,
\label{dGtRwYs}
\end{equation}
which yields an easy expression for $\cos\delta_1$. This implies that $\cos\delta_1$ is constructible, and so is the cyclic quadrangle $P_4$. This settles the case $n=4$. 

Next, let   $n=5$. Note that we know from  Schreiber \cite[Theorem 2 and its proof]{schreiber} that $5\in\setncl 3$; however, we intend to show that  $5\in\setncl 2$. By \eqref{shcrnqltS}, the cyclic pentagon $P_5(1,2,2,2,2)$ exists. By Lemma~\ref{LmMrZT}\eqref{LmMrZTc}, $u=1/(2r)$ is a root of the polynomial $\Wnone 14(1,2,x)$. By computer algebra (or manual computation),  
\[
\begin{aligned}
\Wnone 14(1,2,x)&=16384x^8-8192x^6+1280x^4-63x^2\cr
&=  x^2 \cdot \bigl( 16384x^6-8192x^4+1280x^2-63 \bigr) \text.
\end{aligned}
\]
Since $u\neq 0$, it is a root of  the second factor above. By computer algebra,  this polynomial of degree 6 is irreducible. Thus, Proposition~\ref{szrk}\eqref{szrka} implies that $u$ and, consequently, the cyclic pentagon are non-constructible. Therefore, $5\in\setncl 2$.

Next, let $n=6$, let $0<a,b\in\mathbb R$, $k\in\set{1,2,3}$, $m:=6-k$, and consider the cyclic hexagon \eqref{scABvzPT}. (Note that the order of edges is irrelevant when we investigate the constructibility of cyclic $n$-gons.) 
If $a=b$, then the cyclic hexagon is regular and constructible. If $k=1$, then computer algebra (or manual computation) says that 
\[\Wnone 15(a,b,x)= x\cdot ( -16b^5x^4+20b^3x^2+a-5b);\]
the second factor  is quadratic in $x^2$. Since $u\neq 0$ is a root of the second factor  of  $\Wnone 15(a,b,x)$  by Lemma~\ref{LmMrZT}\eqref{LmMrZTc}, $u^2$, $u$, and the hexagon are constructible. Similarly, 
\[
\begin{aligned}
\Wnone 24(a,b,x)&=
8b^4x^4+(-8b^2-2a^2)x^2+2,\text{  and}\cr
\Wnone 33(a,b,x) &= x\cdot\bigl( (4b^3-4a^3)x^2-3b+3a\bigr),
\end{aligned}
\]
and we conclude the constructibility for $k\in\set{2,3}$ in the same way. Note  that  $\Wnone 33(a,b,x) $ is the zero polynomial if $a=b$; however, this case reduces to the constructibility of the regular hexagon.   Therefore,  the cyclic hexagon is constructible from its side lengths if there are at most two distinct side lengths.

To prove that $6\in\setncl 3$,  we quote the method of  Cz\'edli and  Szendrei~\cite[IX.2.7]{czgsza}; the only difference is that here we choose integer side lengths. Using the cosine angle addition identity, it is easy to conclude that, for all  $\kappa_1,\kappa_2,\kappa_3\in\mathbb R$ such that $\kappa_1+\kappa_2+\kappa_3=\pi$,
\begin{equation}
(\cos\kappa_1)^2+(\cos\kappa_2)^2+(\cos\kappa_3)^2+ 2\cos\kappa_1\cdot\cos\kappa_2\cdot\cos\kappa_3-1=0\,\text.
\label{thrsmm}
\end{equation}
The  cyclic hexagon  $P_6(a_1,\dots, a_6):= P_6(1,1,2,2,3,3)$ exists by  \eqref{shcrnqltS}. 
We will apply 
Proposition~\ref{szrk}\eqref{szrka}.  
Let $\alpha_1,\dots,\alpha_6$ be the corresponding central half angles. Define  $\kappa_1/2=\alpha_1=\alpha_2$, $\kappa_2/2=\alpha_3=\alpha_4$, $\kappa_3/2=\alpha_5=\alpha_6$, and $u=(1/2r)^2$, where $r$ is the radius of the circumscribed circle. 
We have $\cos\kappa_1=\cos(2\alpha_1)=1-2\cdot (\sin\alpha_1)^2= 1-2(a_1/2r)^2=1-2a_1^2u= 1-2u$. We obtain $\cos\kappa_2 =1-8u$ and $\cos\kappa_3 =1-18u$ similarly. 
Since $\kappa_1+\kappa_2+\kappa_3=\pi$, we can substitute these equalities into \eqref{thrsmm}.
Hence, we obtain that $u$ is a root of the cubic polynomial $h_1(x)=144x^3-196x^2+28x-1$. Thus, $2u$ is a root of $h_2(x)= 18y^3-49y^2+14y-1$. Computer algebra says that this polynomial is irreducible. Therefore, 
 $ P_6(1,1,2,2,3,3)$ is not constructible. This completes the proof of Theorem~\ref{thmour} \eqref{thmoura} and \eqref{thmourb}, because the Gauss--Wantzel theorem takes care of $7\in\setncl 1$.
\end{proof}

\begin{proof}[Parts from the proof of Proposition~\ref{prop100} $(${Cz\'edli and Szendrei~\cite{czgsza}}$)$]
Let $n=3$. 
With $d_1=1$, $d_2=2$ and $d_3=3$,  \eqref{thrsmm} and the formulas analogous to \eqref{Tsncnsncnd}
give that $12x^3+14x^2-1=0$. Substituting $x=y/2$, we obtain that $2u=2/r$ is a root of $h_3(y)= 3y^3+7y^2-2$. Since none of $\pm1$, $\pm 2$, $\pm 1/3$ and $\pm 2/3$ is a root of $h_3(y)$, this polynomial is irreducible. Hence, we conclude that the triangle $D_3(1,2,3)$ is not constructible.

Next, following Cz\'edli and  Szendrei~\cite[IX.1.27]{czgsza}, we deal with the cyclic quadrangle $D_4(d_1,\dots, d_4)$, see Figure~\ref{fig1}. Since $\alpha_1+\alpha_2+\alpha_3+\alpha_4=\pi$, we have $\cos(\alpha_1+\alpha_2)=-\cos(\alpha_3+\alpha_4)$. 
Hence, using  the cosine angle addition identity and rearranging and squaring twice, we obtain
\begin{equation}\label{ghit4}
\begin{aligned}
\sum_{j=1}^4 (\cos\alpha_j)^4 &- 2\cdot 
\sum_{1\leq j<s\leq 4} (\cos\alpha_j)^2(\cos\alpha_s)^2\cr
&+4\cdot\cos\alpha_1\cdot \cos\alpha_2\cdot\cos\alpha_3\cdot\cos\alpha_4  \cdot \Bigl(-2+ \sum_{j=1}^4 (\cos\alpha_j)^2\Bigr)
\cr
&+ 4\cdot \sum_{1\leq j<s<t\leq 4}(\cos\alpha_j)^2(\cos\alpha_s)^2(\cos\alpha_t)^2 = 0\text.
\end{aligned}
\end{equation}
Clearly, if we substitute $\cos\alpha_j$ in \eqref{ghit4} by  $d_ju$, for $j=1, \dots, 4$, and divide the equality by $u^4$, then we obtain that $u=1/r$ 
is a root of a polynomial of the form $c_2x^2+c_0$. A straightforward calculation (preferably, by computer algebra) shows that this polynomial is not the zero polynomial since 
\[c_2=4(d_1d_2+d_3d_4)(d_1d_3+d_2d_4)(d_1d_4+ d_1d_3)\text.
\]
Thus $u=1/r$ is constructible, and so is $D_4(d_1,\dots, d_4)$. 

Next, let $n=5$, and let $d_1=d_2=499$, $d_3=d_4=500$ and $d_5=501$; observe that $D_5(d_1,\dots,d_5)$ exists. With $u=1/r$ as before, $\cos(2\alpha_1)= 2(\cos \alpha_1)-1= 2(d_1u)^2-1$, $\cos(2\alpha_3)= 2\cdot(d_3u)^2-1$, and $\cos\alpha_5=d_5u$. Applying \eqref{thrsmm} to $\kappa_1=2\alpha_1$, $\kappa_2=2\alpha_3$, and $\kappa_3=\alpha_5$, we obtain with the help of computer algebra 
that $u$ is a root of the polynomial 
\[1\,494\,006\,000\,000 x^5+498\,005\,992\,004 x^4-5\,988\,012 x^3-1\,995\,995 x^2+6 x+1\text.
\]
Since  this  polynomial is irreducible, $D_5(d_1,\dots,d_5)$  is not constructible . Thus,  $5$ belongs to $\setncd 3$.
\end{proof}

\section{Comments on  Schreiber's argument} \label{moreabout}
Roughly saying, 
``quadratic irrationalities'' are expressions built from their variables and given constants with the help of the four arithmetic operations $+$, $-$, $\cdot$, $/$, and $\sqrt{\phantom n}$; these operations can be used  only finitely many times. 
By \emph{our convention}, to be formulated exactly later in Definition ~\ref{deffffun},   the \emph{domain} of such a function is the largest subset $D$ of $\mathbb R$ such that for all $u\in D$, the expression makes sense in the natural way \emph{without} using complex numbers  and \emph{without} taking limits. 
For example, the domain $\dom(f)$ of the function 
\begin{equation}
f=\sqrt{-1-x^2} + x - \sqrt{-1-x^2}
\label{notabanDon}
\end{equation}
 is empty, while the domain of the function $g(x):=R_6(a_1,\dots, a_6, x)$ given by
\begin{equation}
R_6(a_1,\dots,a_5,x)=\sqrt{a_1}+\dots+\sqrt{a_5}+\sqrt {1/x} - \sqrt{1/(x + x^2)}\text.
\label{lsZgRb}
\end{equation}
is $\dom(g)=(0,\infty)$.

The first  problem with  Schreiber's argument is that quadratic irrationalities are not everywhere continuous in general. It can happen that they are not even defined where \cite{schreiber} needs their continuity. Nothing excludes the possibility that, say, $a_n$ is the denominator of a subterm (or several subterms). This is exemplified by $n=6$ and $R_6$ above with $a_6$ in place of $x$; then $R_6(a_1,\dots,a_{n-1},0)$ is not a meaningful expression, because  $0\notin \dom(g)$. Compare this phenomenon with ``for $a_n = 0$, the quadratic irrationality $R$ describes'' from Section~\ref{ntsect}.
One could argue against us by saying that, as   it is straightforward to see, 
\[\lim_{t\to 0+0} R_6(a_1,\dots,a_5,t)=\sqrt{a_1}+\dots+\sqrt{a_5},
\]
so we could extend the domain of $g$ to contain 0, and then $g$ would be continuous (from the right) at 0 and, what is more important, the limit is again a quadratic irrationality.  However, there are much more complicated expressions than \eqref{lsZgRb}. 
Even if $R_6$ is only an artificial example without concrete geometric meaning, is it always  straightforward  to see that the limit is again a quadratic irrationality? 
As opposed to \cite{schreiber}, we think that this question has to be raised;  for a possible  answer,  see 
the rest of the present paper. 

One could also argue against our strictness at the domain of $f$ from \eqref{notabanDon}; so we note that 
while $\sqrt{\phantom o}$ is a single-valued continuous function on $[0,\infty)\subseteq \mathbb R$,  we know, say, from Gamelin~\cite[page 171]{gamelin} that  
\begin{equation}
\parbox{10.5cm}
{$\sqrt {\phantom o}$  cannot be a single-valued continuous operation on an  open disk  of complex numbers centered at $0$, not even on a punctured disk.}
\label{sqrtnotcontC}
\end{equation}
Hence, complex numbers could create additional problems without solving the problem raised on vanishing denominators like those in \eqref{lsZgRb}.

The second problem is of geometrical nature. Note, however, that this problem is not as important as the first one, because Schreiber does not refer to constructibility programs or similar concepts. Hence, our aim in this paragraph is only to indicate the geometric  background of the  difficulty.  Assume that the cyclic $n$-gon is constructible in general. Take a constructibility program that witnesses this. A step \eqref{constinSrt} can threaten the problem that the ninth and the thirteenth points  are distinct for all $a_n>0$ but they coincide for $a_n=0$, and then they do not determine a line. Then this step does not work for $a_n=0$. Similarly, another step may require to take the intersection of two lines, but if these two lines coincide for $a_n=0$, then they do not determine their intersection point. If so, then this step cannot be a part of a constructibility program. 
Therefore, a constructibility programs that works  for some $n$ may be useless for $n-1$. 

The problems above show that no matter if we use algebraic tools like $R_6$  or geometric tools like constructibility programs, the induction step from $n-1$ to $n$ is not as simple  as \cite{schreiber} seems to expect. On the other hand, 
Remark~\ref{rbhzpGbF} later will show that  the surprising  last sentence of Proposition~\ref{prop100} harmonizes with  Schreiber's argument. However, even Proposition~\ref{prop100}  makes  it desirable to give a precise treatment to Schreiber's idea by determining its scope of applicability.

\section{Basic facts on field extensions}\label{basefieldext}
In this section, the reader is assumed to be familiar with basic field theory. 
We will need $\sqrt{\phantom n}$  as a \emph{continuous} single-valued function. Hence, supported by \eqref{sqrtnotcontC},  
we prefer the field $\mathbb R$ of real numbers to the field $\mathbb C$ of complex numbers in the present paper.
Let $K$ be an abstract field, $c\in K$, and assume that $c$ is distinct from the square of any element of $K$. Denoting by $\sqrt c$ a new symbol that is subject to the computational rule $(\sqrt c)^2=c$,   it is well-known that 
\begin{equation}
K(\sqrt c\,):=\set{a+b\sqrt c:a,b\in F}
\label{sfExt}
\end{equation}
is a field,    a   \emph{quadratic field extension} of $K$. We know that $K$ is a subfield of $K(\sqrt c\,)$ under the natural embedding $a\mapsto a+0\cdot \sqrt c$ . Furthermore, for every $u\in K(\sqrt c\,)$, 
\begin{equation}
\text{there exists a unique }\pair ab\in K\times K \text{ such that }u=a+b\sqrt c\text.
\label{unqcnrp}
\end{equation}
Here, $a+b\sqrt c$ is the so-called  \emph{canonical form} of $u$.
If $c=d^2$ for some $d\in K$, then $K(\sqrt c\,)$  still makes sense but it is $K$ and \eqref{unqcnrp} fails.
By the uniqueness theorem of simple algebraic field extensions, see, for example, Dummit and Foote~\cite[Theorem 13.8, page 519]{dummitfoote}, we have the following uniqueness statement: if $K$ and $K'$ are fields, $\phi\colon K\to K'$ is an isomorphism, $c\in K$ is not a square in $K$, and $c'=\phi(c)$, then 
\begin{equation}
\parbox{11cm}{there exists a \emph{unique} extension $\psi\colon K(\sqrt c\,)\to K'(\sqrt {c'}\,)$ of $\phi$ such that $\psi(\sqrt c\,)= c'$, and $\psi$   is defined by the rule $\psi( a+b\sqrt c\, ) =\phi(a)+\phi(b)\sqrt{c'} $.
}
\label{uniquesmplalg}
\end{equation}

Now, for a subfield $K$ of $\mathbb R$ and $u\in \mathbb R$, we say that  $u$ \emph{is a real quadratic number over} $K$ if there exist an $m\in \mathbb N_0$ and a tower
\begin{equation}
K = K_0\subset K_1\subset \cdots \subset K_m
\label{qxttwR}
\end{equation}
of  field extensions such that $K_j$ is a quadratic extension of $K_{j-1}$ for all $j\in\set{1,\dots, m}$ and $u\in K_m\subseteq \mathbb R$. Note that if $u\in \mathbb R$ is real quadratic over $K$, then it is also algebraic over $K$, but not conversely. Let us emphasize that the concept of  real quadratic numbers does not rely on $\mathbb C$ at all. For example, the equation $a=\sqrt{-6} \cdot \sqrt{-2}$ is not allowed to show that $a$ is  a real quadratic number over $\mathbb Q$.

For a subfield $M$ of $\mathbb R$, $M$ is a \emph{closed with respect to real  square roots} if $\sqrt c\in M$ for all $0\leq c\in M$.
Now let $K$ be a subfield of $\mathbb R$. Using two towers of  quadratic field extensions, see \eqref{qxttwR}, it is routine to check that if $u,v\in\mathbb R$ are real quadratic numbers over $K$, then so are 
$u+v$, $u-v$, $uv$ and, if $v>0$, $u/v$ and $\sqrt v$. Therefore, with the notation 
$\qclo K:=\{u\in \mathbb R: u$ is a real quadratic number over $K\}$,
\begin{equation}
\parbox{7.5cm}
{$\qclo K$ is the smallest subfield of $\mathbb R$ such that includes $K$ and  is closed with respect to real square roots.}
\label{KhaTfj}
\end{equation}
We will call $\qclo K$ the \emph{real quadratic closure} of $K$. 

Next, assume that $a_1,\dots, a_n\in \mathbb R$ define a real number $b=b(a_1,\dots, a_n)$ geometrically, which we want to construct. For example, the $a_i$ can be the side lengths of a cyclic $n$-gon and $b$ can denote the radius of the circumscribed circle of that $n$-gon. In the language of constructibility programs, see around \eqref{constinSrt}, this means that we are given the points $\pair 00$, $\pair 10$, $\pair{a_1}0$, \dots, $\pair{a_n}0$ on the real line, and we want to construct $\pair b0$. In this situation, to ease the terminology, we simply say that we want to \emph{construct a number $b\in \mathbb R$} from $a_1,\dots, a_n\in \mathbb R$. 
By a well-known basic theorem on geometric constructibility, see \cite{czgsza}, 
\begin{equation}
\parbox{9.0 cm}{$b$ is constructible from $a_1,\dots,a_n$ iff $b$ is a  real quadratic number over $\mathbb Q(a_1,\dots, a_n)$, that is, iff $b\in \qclo {\mathbb Q(a_1,\dots, a_n)}$.
}
\label{bdWThrT}
\end{equation}
This statement is a more or less straightforward translation of constructibility programs from  Section~\ref{introsection} to an algebraic language, because the
existence of intermediate points described by the program guarantee that  we do not have to abandon $\mathbb R$ while computing the  coordinates.  Note that, using basic field theory, it is straightforward to deduce   Proposition~\ref{szrk}\eqref{szrkc} from \eqref{bdWThrT}.

Next, we deal with the algebraic background of the situation where one of the parameters in a geometric construction is treated as a variable.
For a subfield $F$ of $\mathbb R$,  let $F[x]$ denote the \emph{polynomial ring} over $F$. It consists of \emph{polynomials}, which are sequences of their coefficients. So a polynomial is just a formal string, not a function, and the same will hold for the elements of the $F_j\bix$ in \eqref{towerxf} below.  However, with any element $f\in F_j\bix$, we will associate a \emph{function $\fun f$} in a natural way. Note that $F$ is a subfield of $F[x]$, because every element of $F$ is a so-called constant polynomial.
The \emph{field of fractions} over $F[x]$ is denoted by $F(x)$.  It consists of formal fractions $f_1/f_2$ where $f_1,f_2\in F[x]$ and $f_2$ is   not the zero polynomial. We say that $f_1/f_2=g_1/g_2$ iff $f_1g_2=f_2g_1$. 
Note that $F$ is a subfield of $F(x)$, because $F[x]$ is a subring of $F(x)$. The element $x\in F(x)$ is  \emph{transcendental over $F$}, and $F(x)$ is a simple transcendental field extension of $F$. If $c\in \mathbb R$ is a transcendental number over $F$, that is, $c$ is a root of no non-zero polynomial with coefficients in $F$, then $F(c)$ is the smallest subfield of $\mathbb R$ including $F\cup\set c$. As a counterpart of \eqref{uniquesmplalg}, the \emph{uniqueness theorem of simple transcendental extensions} asserts that if  $F$ and $F'$ are fields, $\phi\colon F\to F'$ is an isomorphism, $F(c)$ and $F'(c')$ are field extensions such that $c$ and $c'$ are transcendental over $F$ and $F'$, respectively, then 
\begin{equation}
\parbox{11.5cm}{there exists a \emph{unique} extension $\psi\colon F(c)\to F'(c')$ of $\phi$ such that $\psi(c)= c'$;
}
\label{uniquesmpltrans}
\end{equation}
see Dummit and Foote~\cite[page 645]{dummitfoote}. Note that we use $x$ or $y$ for the transcendental element over $F$ and call it an \emph{indeterminate} if we are thinking of evaluating it, but we use $c, d, \dots$ for real numbers that are transcendental over $F\subseteq \mathbb R$. However, say, $F(y)$ and  $F(c)$ in these cases are isomorphic   by \eqref{uniquesmpltrans}; field  theory in itself does not make a distinction between indeterminates and transcendental elements.
A terminological comment: just because we make a distinction between a polynomial $f$ and the function $\fun f$, we call $F(x)$ the \emph{field of polynomial fractions} over $F$ (with indeterminate or variable $x$) but, as opposed to many references, we shall avoid to call it a \emph{function field}.
A \emph{tower of quadratic  field extensions over $F(x)$}  is a finite increasing chain 
\begin{equation}
\parbox{10.5 cm}
{$F(x)=F_0\bix \subset F_1\bix \subset  \cdots\subset F_k\bix $, where $F_j\bix =F_{j-1}\bix (\sqrt{d_j}\,)$  and $d_j\in F_{j-1}\bix $ is not a square in $F_{j-1}\bix $ for $j\in\set{1,\dots,k}$.}
\label{towerxf}
\end{equation}  
Note that  $x$ in the notation $F_k\bix$ reminds us that \eqref{towerxf} starts from $F(x)$ rather than, say, from $F(y)$ or $\mathbb Q$. (We cannot write $F_j[x]$ and $F_j(x)$, because they would denote a polynomial ring of and a transcendent extension over an undefined field $F_j$.)

\begin{definition}\label{deffffun}
Given \eqref{towerxf} and $f\in F_k\bix $, we define a (real-valued) \emph{function $\fun f$ associated with $f$} and its \emph{domain $\dom(\fun f)$} by induction as follows. 
\begin{enumeratei}
\item\label{deffffuna} If $f\in F[x]$ is a polynomial over $F$, then $\fun f$ is the usual function this polynomial determines and $\dom(\fun f)=\mathbb R$. 
\item\label{deffffunb} Assume that  
 $f\in  F_0\bix =F(x)$.  Then 
we write $f$ in the form $f=f_1/f_2$  such that $f_1,f_2\in F[x]$ are relatively prime polynomials; this is always possible and the roots of $f_2$  are uniquely determined. We let 
 $\dom(f)=\mathbb R\setminus\{$real roots of $f_2\}$,
and let $\fun f$ be the function defined by the rule $\fun f(r)=\fun{f_1}(r)/ \fun{f_2}(r)$ for $r\in\dom(\fun f)$.
\item\label{deffffunc} Assume $j\geq 1$ and that $\ffun{d_j}$ and $\dom(\ffun{d_j})$ have already been defined. 
 We let $ \dom(\varfun{\sqrt{d_j}\,}) =\set{r\in\dom(\ffun{d_j}):  \ffun{d_j}(r)\geq 0 }$. For $r\in \dom(\varfun{\sqrt{d_j}\,})$, let $\varfun{\sqrt {d_j}\, } (r) $ be the unique non-negative real number whose square is $\ffun{d_j}(r)$.
\item\label{deffffund} Assume that $f\in F_j\bix $. By  \eqref{unqcnrp},  there are unique $f_1,f_2$ in $F_{j-1}\bix $ such that $f=f_1+f_2\sqrt{d_j}$. We let $\dom(\fun f)=\dom(\fun{f_1})  \cap \dom(\fun{f_2})\cap   \dom(\varfun{\sqrt{d_j}\,})$ and, for $r\in \dom(\fun f)$, $\fun f(r):= \fun{f_1}(r)+\fun{f_2}(r)\cdot \varfun{\sqrt{d_j}\,}(r)$.
\end{enumeratei}
\end{definition}

Two functions are considered \emph{equal} if they have the same domain and they take the same values on their common domain. Usually, $f_1\neq f_2\in F_k\bix $ does not imply $\fun{f_1}\neq\fun{f_2}$. For example, $\dom(\varfun { \sqrt{-1-x^2 }  }) =\emptyset=\dom(\varfun { \sqrt{-1-x^4 }  }) $ and  $ \varfun { \sqrt{-1-x^2 }  } = \varfun { \sqrt{-1-x^4 }  }$, but  $ \sqrt{-1-x^2 }\neq  \sqrt{-1-x^4 }$.  This explains why we make a notational distinction between $f$ and $\fun f$ in general. Note that for a polynomial  $f\in F[x]\subseteq F(x)=F_0(x)$, especially for $f(x)=x^n$, this distinction is not necessary, and we are not always as careful as in this section. That is, 
\begin{equation}
\parbox{10.3cm}{We often write  $f(c)$ for $c\in F$ rather than $\fun f(c)$, if $f$ is a polynomial or, in later sections, if $f=f_1/f_2$ where $f_1$ and $f_2$ are polynomials.}
\label{oFnWrT}
\end{equation}
Due to the following lemma, which will often be applied without referring to it, the distinction between $f$ and $\fun f$  will not cause difficulty.

\begin{lemma}\label{szddLm}
Given \eqref{towerxf} and $f,g\in F_k\bix $, let $r\in\dom(\fun f)\cap\dom(\ffun g)$. Then 
$\varfun{f+g}(r)=\fun f(r)+\ffun g(r)$ and $\varfun{fg}(r)=\fun f(r)\cdot \ffun g(r)$. 
\end{lemma}

\begin{proof} For $k=0$, the statement is trivial. The induction step for the product runs as follows. Assume that $f,g\in F_j\bix$. To save space,  let $z=\varfun{\sqrt{d_j}}(r)$. Using Definition~\ref{deffffun},  the induction hypothesis, and $z^2 =\fun{d_j}(r)$, we obtain that 
\allowdisplaybreaks{
\begin{align*}
\varfun{f g}&(r)
=
\bvarfun{(f_1+f_2\sqrt{d_j}\,)  (g_1+g_2\sqrt{d_j}\,)  }(r)
\cr
&= \bvarfun{ f_1g_1 + f_2g_2d_j + (f_1g_2+f_2g_1 ) \sqrt{d_j}\,) } (r)  \cr
&= \varfun{ f_1g_1 + f_2g_2d_j}(r) + \varfun{f_1g_2+f_2g_1}(r)\cdot  z  \cr
&=\fun{f_1}(r)\, \ffun{g_1}(r)+\fun{f_2}(r)\, \ffun{g_2}(r) \, \ffun{d_j}(r)   +\fun{f_1}(r)\, \ffun{g_2}(r)z+ \fun{f_2}(r)\, \ffun{g_1}(r)z\cr
&=(\fun{f_1}(r) +   \fun{f_2}(r)z)  \cdot  (\ffun{g_1}(r) +   \ffun{g_2}(r)z) = \fun f(r)\cdot \ffun g(r)\text.
\end{align*}
}%
The evident treatment for addition is omitted.
\end{proof}

\begin{lemma}\label{infntroots}
For a subfield $F\subseteq \mathbb R$ and 
$f\in F_k\bix $,  assume that   $\fun f$ has infinitely many roots in $\dom(\fun f)$.  Then $f=0$ in $F_k\bix $ and $\fun f\colon \mathbb R\to\set 0$ with $\dom(\fun f)=\mathbb R$. 
\end{lemma}

\begin{proof}[Proof of Lemma~\ref{infntroots}] 
We use induction on $k$. 
First, assume that $k=0$. Then $f=f_1/f_2\in F(x)$ where $f_1,f_2\in F[x]$ are relatively prime polynomials and $f_2$ is not the zero polynomial. Necessarily, $f_1$ is the zero polynomial (with no nonzero coefficient), because otherwise $\fun{f_1}$ and $\fun f$ would only have finitely many roots. Therefore, in the field $F(x)$,  $f_1$ and $f=f_1/f_2$ are the zero element, as required.
 
Next, assume that $k>0$ and the lemma holds for $k-1$. Using \eqref{sfExt},  \eqref{unqcnrp}, and the notation given in \eqref{towerxf}, we obtain that  there are unique $f_1, f_2\in F_{k-1}\bix $ such that $f=f_1 +f_2\sqrt{d_k}$. We can assume that $f_2\neq 0$, because otherwise $f\in F_{k-1}\bix $ and the induction hypothesis applies. 
Let $\overline f= f_1 -f_2 \sqrt{d_k }\in F_k\bix $, and define $g :=f \cdot \overline f =  f_1^2-f_2^2 d_k$. 
Since $\dom(\ffun g)\supseteq \dom(\fun f)=\dom(\ffun{\overline f})$, if   $y\in\dom(\fun f)$ is a root of  $\fun f$, then $\ffun g(y)=\fun f(y)\,\ffun{\overline f}(y)= 0\cdot \ffun{\overline f}(y)=0$ shows that $\ffun g(y)=0$. Thus,  $\ffun g$ has infinitely many roots. On the other hand,  $g \in F_{k-1}\bix $, so  the induction hypothesis gives that $g=0$ in   $F_{k-1}\bix $. Therefore, $f_1^2 = f_2^2d_k$ and $d_k$ is the square of $f_1/f_2\in F_{k-1}\bix $, contradicting \eqref{towerxf}.
\end{proof}

\begin{corollary}\label{corTHSM}
Using the notation of \eqref{towerxf},  assume that  $g_1, g_2\in  F_k\bix $ 
 such that  
$\dom(\ffun{g_1})\cap\dom(\ffun{g_2})$ is infinite. Then $g_1=g_2$  if and only if $\ffun{g_1}=\ffun{g_2}$.
\end{corollary}

\begin{proof} Apply Lemma~\ref{infntroots} for $f:=g_1-g_2\in F_k\bix $.
\end{proof}

\section{Power series with dyadic exponents} \label{powsersect}
In this section, the references aim at well-known facts from calculus and complex analysis; most readers do not need these outer sources. However, some minimum knowledge of calculus is assumed.  Because of \eqref{sqrtnotcontC},  we mainly study real numbers.  A \emph{strict right neighborhood of $0$} is an open interval $(0,\epsilon)$ where $0<\epsilon \in\mathbb R$. The adjective ``strict'' is used to emphasize that a strict right neighborhood of 0 does not contain 0. 
Since $x^{2^{-k}}$ for $k\in\mathbb N$ is not defined if $x<0$, we only consider strict right neighborhoods of 0.    A \emph{dyadic number} is  a rational number of the form $a\cdot 2^t$ where $a,t\in\mathbb Z$. 
Subfields of $\mathbb R$ closed with respect to real  square roots were defined right before \eqref{KhaTfj}. 
The goal of this section is to prove the following theorem.

\begin{theorem}\label{qdrthM}  
Let $F$ be a subfield of $\mathbb R$ such that $F$ is closed with respect to real square roots. Also, let $k\in\mathbb N_0$, and consider a tower \eqref{towerxf} of quadratic  field extensions of length $k$ over $F(x)$. Finally, let $f\in F_k\bix $, and assume that there is a strict right neighborhood $(0,\epsilon)$ of $\,0$ such that  $(0,\epsilon)\subseteq\dom(\fun f)$.  Then there exist an integer $t\in \mathbb Z$ and elements $b_t, b_{t+1}, b_{t+2},\dots$ in the field $F$ such that 
\begin{equation}
\fun f(x) = \sum_{j=t}^\infty b_j\cdot x^{j\cdot 2^{-k}} \label{lsdTbnQwlR}
\end{equation}
holds in some  strict right neighborhood of 0. Furthermore, 
if $f\neq 0$, then $b_t\neq 0$ and $t,b_t, b_{t+1},b_{t+2},\dots$ are uniquely determined.
\end{theorem} 

Let us emphasize that $t$ in \eqref{lsdTbnQwlR} can be negative. Before proving this theorem, we need  a  lemma;  $f$ below  has nothing to do with $F_k\bix $.

\begin{lemma}\label{kPlemma}
Assume that $0<\epsilon\in\mathbb R$,    $f\colon (0,\epsilon)\to \mathbb R$ is a  non-negative function, and  
\begin{equation}
f(x)=1+\sum_{j=1}^\infty a_j x^{j}\qquad(\text{with real coefficients } a_j)
\label{aPwrser}
\end{equation}
for all $x\in(0,\epsilon)$. Let $b_j$ denote the  the real numbers defined recursively by 
\begin{equation}
b_j=\begin{cases}
{\displaystyle{a_j/2 - \sum_{t=1}^{\lfloor\kern 0.5 pt j/2 \rfloor}  b_t b_{j-t}  }}, &\text{ if } j\text{ is odd,}\cr
{\displaystyle{a_j/2 - b_{j/2}^2   - \sum_{t=1}^{j/2 -1}  b_t b_{j-t}  }}, &\text{ if } j\text{ is even,}
\end{cases}
\label{recursFm}
\end{equation}
for $j\in \mathbb N$. Then, in an appropriate  strict right neighborhood of $0$, 
\begin{equation}
\sqrt{f(x)}=1+\sum_{j=1}^\infty b_j x^{j}\text{.}
\label{bSqnCe}
\end{equation}
\end{lemma}

By basic properties of the radius of convergence, see, e.g.,  Rudin~\cite[Subsection 10.5 in page 198]{rudin},  the series in \eqref{aPwrser} is also convergent  in   $(-\epsilon,\epsilon)$; however, the function $f(x)$ need not be defined for $x\in(-\epsilon,0]$.

\begin{proof}[Proof of Lemma~\ref{kPlemma}]
First, we show the existence of a power series that represents $\sqrt f$ in the sense of \eqref{bSqnCe}, but we do not require the validity of \eqref{recursFm} at this stage. We know that, for $x\in (-1,1)$, the \emph{binomial series} 
\begin{equation}
\sum_{j=0}^\infty {1/2 \choose j} x^j = 1 +  \frac{\frac 12}{1!}  x + \frac{\frac 12(\frac{1}{2}-1) }{2!} x^2
+ \frac{\frac 12 (\frac{1}{2}-1 ) (\frac{1}{2}-2)   }{3!} x^3 + \cdots
\label{binomseR}
\end{equation}
is absolutely convergent and it  converges to $\sqrt{1+x}$ on $(-1,1)$;  see, for example, Wrede and Spiegel~\cite[page 275]{schaumreal}.  Therefore, the same series defines a holomorphic function $g(z)$ on the open disk $D_1=\set{z\in \mathbb C: |z|<1}$ of complex numbers.  
Note that, for $z\in D_1$,   $g(z)$ is one of the complex values of $\sqrt {1+z}$. 
Since both the series \eqref{binomseR} and the function $\sqrt{1+x}$ are continuous on $(-1,1)$ and they take the same \emph{positive} value at $x=0$, we obtain that $g(x)=\sqrt{1+x}$ for any real $x\in(-1,1)$.
Observe that  $\sum_{j=1}^\infty a_jx^j$ in \eqref{aPwrser} is convergent and differentiable on the open disk $D_\epsilon=\set{z\in \mathbb C: |z|<\epsilon}$. By its continuity,   $|\sum_{j=1}^\infty a_jx^j|<1$ on an appropriate small open disc $D_\delta$, where $\delta<\epsilon$. Furthermore, as any power series within its radius of convergence, $\sum_{j=1}^\infty a_jx^j$ is differentiable on $D_\delta$.
Therefore, the composite function $g(\sum_{j=1}^\infty a_jx^j)$  is also differentiable on $D_\delta$.  Hence, by a basic property of holomorphic functions, 
there exists a power series $1+\sum_{j=1}^\infty b_jx^j$,  without stipulating \eqref{recursFm}, such that $g(\sum_{j=1}^\infty a_jx^j) =  \sum_{j=0}^\infty b_jx^j =  1+\sum_{j=1}^\infty b_jx^j$  for all complex $x\in D_\delta$; see, for example, Rudin~\cite[Theorem 10.16 in page 207]{rudin}. 
Here,  $b_0=1$ follows from $g(0)=1$.  
 In particular, for all real $x$ in a small strict right neighborhood of $0$, 
\[
\sqrt{f(x)}
= \textstyle{ \sqrt{ 1 +  \sum_{j=1}^\infty a_jx^j} } = g(\sum_{j=1}^\infty a_jx^j)= 1+\sum_{j=1}^\infty b_jx^j \text.
\]
This proves the existence of an appropriate power series such that  \eqref{bSqnCe} holds in a strict right neighborhood of $0$, but we still have to show the validity of \eqref{recursFm}.

Since the power series in \eqref{bSqnCe} is absolute convergent in a small strict right neighborhood of $0$, its product with itself  converges to $\bigl(\sqrt {f(x)}\,\bigr){}^2=f(x)$; see, for example,  Wrede and Spiegel~\cite[Theorem 5 in Chapter 11, page 269]{schaumreal}. Therefore,
\begin{equation}
\begin{aligned}
f(x)& =1 + (b_1+b_1)x + (b_2+b_1b_1+b_2)x^2 \cr
&+(b_3+b_1b_2+b_2b_1+b_3)x^3\cr
&+(b_4+b_1b_3+b_2b_2+b_3b_1+b_4)x^4
\cr
&+(b_5+b_1b_4+b_2b_3+b_3b_2+b_4b_1+b_5)x^5\cr
&+(b_6+b_1b_5+b_2b_4+b_3b_3+b_4b_2+b_5b_1+b_6)x^6+\cdots\,\,.
\end{aligned}
\label{dkjGtWrB}
\end{equation}
The power series converging to a function is unique; see, for example, Rudin~\cite[Corollary to Theorem 10.6 in page 199]{rudin}. Consequently, comparing the coefficients in \eqref{dkjGtWrB} with those in \eqref{aPwrser}, we obtain that \eqref{recursFm} holds.
\end{proof}

\begin{proof}[Proof of Theorem~\ref{qdrthM}] 
First, to prove the uniqueness part, assume that \eqref{lsdTbnQwlR} holds for all $x\in (0,\epsilon)$. Substituting $\xi ^{ 2^k}$ for $x$, we obtain that,    for all $\xi\in (0, \epsilon^{2^{-k}})$,  $\fun f(\xi^{2^k}) =  \sum_{j=t}^\infty b_j \xi^j  =  \sum_{s=0}^\infty b_{s+t} \xi^{s+t} $. There are two ways to continue. First, we can consider the function $g(z):=\fun f(z^{2^k})$ of a complex variable $z$ and then we can refer to the uniqueness of its Laurent series expansion; e.g., see Gamelin~\cite[page 168]{gamelin}. Second, and more elementarily, we can observe that 
 $\xi^{-t}\cdot \fun f(\xi^{2^k}) = \sum_{j=0}^\infty  b_{j+t}\cdot \xi^j $  for all $\xi\in (0, \epsilon^{2^{-k}})$. However, the power series of a function is unique; see, e.g., Rudin~\cite[Corollary to Theorem 10.6 in page 199]{rudin}. The uniqueness of the coefficients in $\sum_{j=0}^\infty  b_{j+t}\cdot \xi^j $ implies 
the uniqueness part of the theorem.

We prove the existence part of the theorem by induction on $k$. If $k=0$, then $f=f_1/f_2$ for some relatively prime polynomials $f_1,f_2\in F[x]$. If $\fun {f_2}(0)\neq 0$, then there exists an open circular disc $D$ of positive radius with center 0 such that $\fun{f_2}$ has no zeros in $D\subset \mathbb C$.
Since $\fun f$ in $\dom(\fun f)$ equals $\fun{f_1} /\, \fun{f_2}$ and both the numerator and the denominators are polynomial functions,  $\fun f$  is holomorphic in $D$ and its Taylor series, which is of form \eqref{lsdTbnQwlR} with $t=k=0$, converges to $\fun f(x)$ by a well-known property of holomorphic functions; see, e.g.,  Rudin~\cite[Theorem 10.16 in page 207]{rudin}.  Furthermore, since the coefficients of the Taylor series are obtained by repeated derivations, they belong to $F$. The other case, where $\fun{f_2}(0)=0$, reduces the above case easily as follows. Assume that   $\fun{f_2}(0)=0$. Then there are a unique $n\in \mathbb N$ and a unique  $g\in F[x]$ such that $f_2(x)=x^n \cdot g(x)$ and $\ffun g(0)\neq 0$. By the previous case, $\varfun{ f_1/g}$ is represented by its Taylor series with coefficients in $F$. We can multiply this series by $x^{-n}$ componentwise, see for example, Gamelin~\cite[page 173]{gamelin}. 
In other words, we multiply the series with  $\varfun{x^{-n}}$. 
In this way, we obtain the Laurent series of
 $ \varfun{f_1/g} \cdot x^{-n} =   \varfun{(f_1/g) \cdot x^{-n}}   =       \fun f$ 
with the same coefficients but ``shifted'' to the left by $n$. Thus,  we obtain a series \eqref{lsdTbnQwlR} with $t=-n$ and $k=0$, which converges to $\fun f$ in a strict right neighborhood of $0$, as required.

Next, we assume that $k>0$ and that the theorem holds for $k-1$. Combining \eqref{sfExt} and \eqref{towerxf}, we obtain that $f$ is of the form $f=f_1+f_2\sqrt{d_k}$ where $f_1, f_2, d_k\in F_{k-1}\bix $. Since the validity of the existence part of the theorem is obviously inherited from the summands to their sum when we add   two functions, it suffices to deal with $f_2\sqrt{d_k}$ in case  $f_2$ in  $F_{k-1}\bix $ is distinct from $0$. Note at this point that  the validity of the existence part of the theorem for $f_2$  is also inherited if we change $f_2$ to $-f_2$. 
We know from Lemma~\ref{infntroots} that $\fun{f_2}$ and $\ffun{d_k}$ only have finitely many roots in their domains. Therefore, decreasing the strict right neighborhood of $0$  if necessary, we can assume that 
$\varfun{f_2{}^2\cdot d_k}$ has no root in $(0,\epsilon)$ at all. 
Since this function is composed from the four arithmetic operations and square roots, it is continuous on its domain. Hence, either $\fun{f_2}$ is positive on $(0,\epsilon)$, or it is negative on  $(0,\epsilon)$. 
We can assume that $\fun{f_2}$ is positive on $(0,\epsilon)$, because in the other case we could work with $\varfun{-f_2} = -\fun{f_2}$ similarly. 
Also, $\ffun{d_k}$ is nonnegative on $(0,\epsilon)$, because  $(0,\epsilon)\subseteq \dom(\fun f)$. 
Hence $\varfun{f_2\sqrt{d_k}\,}$ and $\varfun{\sqrt{f_2{}^2\cdot d_k}}$ agree on the interval $(0,\epsilon)$.   Since we are only interested in our function on $(0,\epsilon)$, it suffices to deal with $\varfun{\sqrt{f_2{}^2\cdot d_k}}$ rather than $\varfun{f_2\sqrt{d_k}\,}$.   By the induction hypothesis, $\varfun{f_2{}^2d_k}$ can be given in  form \eqref{lsdTbnQwlR}, with $k-1$ in place of $k$. To simplify the notation, we will write $f$ rather than $f_2{}^2d_k$. So, $f\in F_{k-1}\bix$ and $\ffun f$ is positive on a strict right neighborhood of 0.  By the induction hypothesis, we have a unique series below with all the $c_j$ in $F$ and $c_t\neq 0$ such that 
\begin{equation}
\fun f(x) = \sum_{j=t}^\infty c_j\cdot x^{j\cdot 2^{1-k}}
\label{lsdpTindhQR}
\end{equation}
in a strict right neighborhood of $0$. Note that 
\eqref{lsdpTindhQR} may fail for $x=0$. 
Taking out the first summand from  \eqref{lsdpTindhQR}  and writing $s$ for $j-t$, we obtain that 
\begin{equation*}
\fun f(x) =  c_t\cdot x^{t\cdot{2^{1-k}}} \cdot   \sum_{s=0}^\infty \frac{c_{s+t}}{c_t}\cdot x^{s\cdot 2^{1-k}} = c_t\cdot x^{t\cdot{2^{1-k}}} \cdot  \Bigl(1+ \sum_{s=1}^\infty a_s y^s \Bigr),
\label{dkuTrG}
\end{equation*}
where $a_s:=c_{s+t}/c_t\in F$ and $y:= x^{2^{1-k}} =  x^{2\cdot 2^{-k}} $. 
The rightmost infinite sum above is convergent in a strict right neighborhood of 0, because so is \eqref{lsdpTindhQR}.  Applying Lemma~\ref{kPlemma}  with  $b_0:=1$ and $b_j$ defined by  \eqref{recursFm} and replacing $t+2j$ by $s$ in the next step, we obtain
\begin{equation}
\begin{aligned}
\varfun {\sqrt f\,}(x) = 
\sqrt{\fun f(x)} &=  \sqrt{c_t}\cdot x^{t\cdot{2^{-k}}} \cdot \sum_{j=0}^\infty b_j y^j = 
 \sum_{j=0}^\infty b_j \sqrt{c_t}  \cdot   x^{(t+2j)\cdot{2^{-k}}} \cr
 & =
 \sum_{s=t}^{\infty,\bullet}  b_{(s-t)/2}\cdot \sqrt{c_t}  \cdot   x^{s\cdot{2^{-k}}}   =
\sum_{s=t}^\infty d_s \cdot   x^{s\cdot{2^{-k}}} 
\end{aligned}
\end{equation}
in some right neighborhood of 0, 
where $\bullet$ means that only those subscripts $s$ occur for which  $s\equiv t$ (mod 2), and we have that 
$d_s:=b_{(s-t)/2}\cdot \sqrt{c_t} $ for $s\equiv t$ (mod 2), and $d_s:=0$ otherwise.
Since the $a_s$ are all in $F$, the $b_j$ given by \eqref{recursFm}  are also in $F$. Since $\ffun f$ is positive in a (small) strict right neighborhood of 0 and, as $x$ tends to 0, the series in \eqref{lsdpTindhQR} is dominated by $c_t\cdot x^{t\cdot 2^{1-k}}$, we obtain in a straightforward way that $c_t>0$. Hence, $\sqrt{c_t}\in F$, because $F$ is closed with respect to real square roots, and we conclude that $d_j\in F$ for $j\in\set{t,t+1,t+2\dots}$.  
This completes the inductive step and the proof of Theorem~\ref{qdrthM}.
\end{proof}

\section{Puiseux series and historical comments}\label{sectpuiseuxcomments}
Now, we are going to compare Theorem~\ref{qdrthM} to known results; the rest of the paper does not rely on this section. 
Let  $F$ be a subfield of the field $\mathbb C$ of complex numbers.
 A \emph{Puiseux series} over $F$ is a generalized power series of the form 
\begin{equation}
\sum_{j=t}^\infty b_j\cdot x^{j/m},
\label{sldTnGR}
\end{equation} 
where   
$m\in \mathbb N$, $t\in \mathbb Z$, and  $b_j\in F$ for $j\in\set{t, t+1,t+2, \dots}$.   In other words, a Puiseux series is obtained from a Laurent series with finitely many powers of negative exponent by substituting ${\root m\of x}$ for its variable.  Note that, for $w\in\mathbb C\setminus\set 0$, the substitution $\sum_{j=t}^\infty b_j\cdot w^{j/m}$ of $w$ for $x$ in \eqref{sldTnGR} is understood such that 
first we fix one of the $m$ values of $\root m \of w$, and then we use \emph{this} value of $\root m \of w$ to define $w^{j/m}$ as $\bigl(\kern-2pt \root m \of w\,\bigr)^j$, for all $t\leq j\in\mathbb Z$. This convention allows us to say that certain Puiseux series are convergent in a punctured disk $D_\epsilon^+=\set{z\in \mathbb C: 0<|z|<\epsilon}$.
This concept and  the following theorem go back to  
Isaac Newton~\cite{newton,wiki-puiseux}.

\begin{theorem}[{Puiseux's Theorem, \cite{puiseux1, puiseux2}; see also \cite{nowak, ruiz, wiki-puiseux}}]\label{puiseux}      
Let $F\subseteq \mathbb C$ be a field, and let 
\begin{equation} 
P(x,y)=A_0(x)+A_1(x)y+\cdots+A_n(x) y^n
\label{pZRSrS}
\end{equation}
be an irreducible polynomial in $F[x,y]=F[x][y]$ such that the $A_j[x]$ belong to $F[x]$ and $A_n(x)\neq 0$. 
 If $F$ is  algebraically closed, then  there exist a  \emph(\kern-1pt small\emph) positive $\epsilon\in \mathbb R$  and a Puiseux series \eqref{sldTnGR} such that this series converges to a function $Y(x)$ in the punctured disk $D_\epsilon^+$ and, for all $x\in D_\epsilon^+$, we have that  $P(x,Y(x))=0$.
\end{theorem}

Note that $Y(x)$ above is a multiple-valued function in general and it is rarely continuous in $D_\epsilon^+$. This is 
exemplified by $P(x,y)=y^2-x\in\mathbb C[x][y]$ together with  \eqref{sqrtnotcontC}, where  the Puiseux series is the one-element sum~$ x^{1/2}$. 
Note also that Theorem~\ref{puiseux}   seems not to imply Theorem~\ref{qdrthM} in a straightforward way, because $F_k\bix $ is not an algebraically closed field for a finitely generated field $F$. Actually, Theorem~\ref{puiseux} only implies a weaker form of Theorem~\ref{qdrthM}. This weaker form asserts that  the $b_j$ belong to the algebraic closure of $F(x)$;  this statement  is useless at geometric constructibility problems. Fortunately, as our proof witnesses, it was possible to tailor Puiseux's proof to the peculiarities of Theorem~\ref{qdrthM}.

\section{A limit theorem for geometric constructibility}\label{limthemsect} 
The aim of this section is to prove the following statement. For the constructibility of numbers, see \eqref{bdWThrT}.

\begin{theorem}[Limit Theorem for Geometric Constructibility]\label{limthm}
Let  $a_1$, \dots, $a_m$, and $d$ be real numbers such that  $d$ is geometrically constructible from   $a_1,\dots, a_m$. Let $u(x)$ be a real-valued function. 
If there exist a  positive $\epsilon\in\mathbb R$ and a nonzero  polynomial $W$ of two indeterminates over $\mathbb R$ such that 
\begin{enumeratei}
\item\label{limthma} $u(x)$ is defined on the interval $(d,d+\epsilon)$, 
\item\label{limthmb} every coefficient in $W$ is geometrically constructible from   $a_1,\dots, a_m$, and $W(c,u(c))=0$ holds for all $c\in (d,d+\epsilon)$ such that $c$ is transcendental over $\mathbb Q(a_1,\dots, a_m)$, 
\item\label{limthmc} for all $c\in (d,d+\epsilon)$ such that $c$ is transcendental over $\mathbb Q(a_1,\dots, a_m)$, $u(c)$ is geometrically   constructible from $c, a_1,\dots, a_m$, and
\item\label{limthmd} $\lim_{x\to d+0} u(x)\in \mathbb R$ exists,
\end{enumeratei}
then $\lim_{x\to d+0} u(x)$ is geometrically  constructible from $a_1,\dots, a_m$.  
\end{theorem}

Clearly, if  we require  the equality and the constructibility in \eqref{limthmb} and \eqref{limthmc}, respectively, for all elements $c$ of the interval $(d,d+\epsilon)$, not only for the transcendental ones,  then we obtain a corollary with simpler assumptions. However, we shall soon see that Theorem~\ref{limthm} is more useful than its corollary just mentioned.

\begin{remark} To show that \eqref{limthmb} in Theorem~\ref{limthm} is essential, we have the following example. As usual, $\lfloor\,\,\rfloor$ will stand for the (lower) integer part  function.
Let $\epsilon=1$, $d=0$, and  $m=0$.
For $x\in(0,1)$, let $j(x)=-\lfloor\log_{10}(x)\rfloor\in \mathbb Z$. We define $u(x)$ by   $u(x):= \lfloor 10^{j(x)}\cdot \pi \rfloor \cdot 10^{-j(x)}$, where $e=  3.141\, 592\, 653\, 589\, 793\dots$, as usual.
For example,  if  $ x=\pi/10^7= 0.000\,000\,314\, 159\, \dots$, then  
$j(x)= 7$ and $u(x)= 3.141\, 592\, 6\in\mathbb Q$. Clearly,  $\lim_{x\to 0+0} u(x)=\pi\in\mathbb R$ and, for every $c\in (0,1)$, $u(c)\in \mathbb Q$ is constructible.  Although $u(c)$ is the root of the rational polynomial $x-u(c)$ with constructible coefficients, this polynomial  depends on $u(c)$.  All assumptions of Theorem~\ref{limthm} hold except  \eqref{limthmb}. 
If the theorem held without assuming  \eqref{limthmb}, then $\pi$ would be geometrically constructible over $\mathbb Q$, which would contradict \eqref{KhaTfj} by Lindemann~\cite{lindemann}. 
Alternatively, it would  be a contradiction because Squaring the Circle impossible.
Therefore, \eqref{limthmb} cannot be omitted from  Theorem~\ref{limthm}.
\end{remark}

\begin{remark}\label{whGnCmeans} 
Consider the function $u\colon (0,1)\to\mathbb Q$ from the previous remark. For each $c\in(0,1)$, $u(c)$ is constructible from $0$, $1$, and $c$. However,   based on Theorem~\ref{limthm}, it is not hard to prove that    $u(x)$ is not constructible from $0$, $1$, and $x$ in the sense of \eqref{pzCrVy}; the details of this proof are omitted. This example points out that the meaning of ``in general not constructible'' needs a definition; an appropriate definition will be given in Remark~\ref{schRfnl}.
\end{remark}

\begin{remark}\label{rbhzpGbF} 
The assumption that  $d$ is  geometrically  constructible from   $a_1,\dots, a_m$  cannot be omitted from Theorem~\ref{limthm}. 
\end{remark}

\begin{proof}[Proof of Remark~\ref{rbhzpGbF}] Our argument is based on   Proposition~\ref{prop100}. Suppose, for a contradiction, that Theorem~\ref{limthm}  without assuming that $d$ is constructible, referred to as the \emph{forged theorem}, holds.   Let $d_1,d_2,d_3$ be arbitrary positive real numbers such that $D_3(d_1,d_2,d_3)$ exists. Note that $\tuple{3,d_1,d_2,d_3}$ will play the role of $\tuple{m,a_1,\dots,a_m}$ in the forged theorem.  Denote by  $r(d_1,d_2,d_3)$  the radius of the circumscribed circle of $D_3(d_1,d_2,d_3)$, and let $d:=- r(d_1,d_2,d_3)$.  By geometric reasons,  there is a positive $\epsilon\in\mathbb R$ such that 
 $D_4(d_1,d_2,d_3, -x)$ exists for all $x\in (d,d+\epsilon)$; just think of an infinitesimally small fourth side. Let $u(x)$ denote  
the radius of the circumscribed circle of $D_4(d_1,d_2,d_3, -x)$. We know from  Proposition~\ref{prop100} that  $D_4(d_1,d_2,d_3, -c)$ is constructible from $d_1,d_2,d_3,c$, provided $c\in (d,d+\epsilon)$. Hence, $u(c)$ is also constructible and   \eqref{limthmc} of the forged theorem holds.  It is straightforward to conclude from  \eqref{dGtRwYs} that so does \eqref{limthmb}, while the satisfaction of \eqref{limthma} is evident. Since the geometric dependence of $u(x)$ on $x$ is continuous,  $\lim_{x\to d+0} u(x) =r(d_1,d_2,d_3)\in\mathbb R$. So, \eqref{limthmd} is also satisfied. Hence, by the forged theorem, $r(d_1,d_2,d_3)=\lim_{x\to d+0} u(x)$ is geometrically constructible from $d_1,d_2,d_3$. Therefore, $D_3(d_1,d_2,d_3)$ is also  constructible from $d_1,d_2,d_3$. However, this contradicts Proposition~\ref{prop100}, completing the proof of Remark~\ref{rbhzpGbF}.
\end{proof}

We will derive Theorem~\ref{limthm} from the following, more technical statement.

\begin{proposition}\label{limprop}
Let $F$ be a subfield of $\,\mathbb R$ such that  $F$ is closed with respect to real square roots, and let $d\in F$. Assume that  $W\in F[x,y]\setminus \set 0$ is a nonzero polynomial. If $\tuple{c_j:j\in\mathbb N}$  and $\tuple{u_j:j\in\mathbb N}$ are sequences of real numbers such that 
\begin{enumeratei}
\item $c_j$ is transcendental over $F$ and $d<c_j$,  for all $j\in \mathbb N$, 
\item for all $j\in \mathbb N$, $u_j$   is a real quadratic number over $F(c_j)$,  see \eqref{qxttwR},
\item  $W(c_j,u_j)=0$, for all $j\in \mathbb N$, and
\item $\lim_{j\to\infty} c_j=d$ and   $\lim_{j\to\infty} u_j\in \mathbb R$,
\end{enumeratei}
then $\lim_{j\to\infty} u_j\in F$. 
\end{proposition}

As a preparation to the proof of Proposition~\ref{limprop}, we need two lemmas. In the first of them, all sorts of intervals, for example,  $[a,b)$, $(a,\infty)$, $(-\infty,b]$, $(a,b)$, etc., are allowed, and the empty union is the empty set.

\begin{lemma}\label{lMMdom}
If  $F$ is a subfield of $\mathbb R$, $k\in\mathbb N$, $F_k\bix $ is as in  \eqref{towerxf}, and $f\in F_k\bix $, then $\dom(\fun f)$ is the union of finitely many intervals and $\fun f$ is continuous on each of these intervals.
\end{lemma}

\begin{proof} We use  induction on $k$. 
If $f\in F_0(x)=F(x)$, then $f=f_1/f_2$ where $f_1,f_2\in F[x]$ are polynomials. Hence, $\fun{f_2}$ only has finitely many (real) roots and the statement for $k=0$ is clear. Next, assume that $k>0$ and that the lemma holds for $k-1$. Lemma~\ref{infntroots} gives that $\ffun{d_k}$ has only finitely many roots. By the induction hypothesis,  $\dom(\ffun{d_k})$ is the union of finitely many intervals and $\ffun{d_k}$ is continuous on each of these intervals. Thus, it follows from continuity that $\dom(\varfun{\sqrt{d_k}\,})$ is the union of finitely many intervals; see Definition~\ref{deffffun}\eqref{deffffunc}. 
Furthermore, by the continuity of $\sqrt{\phantom o}$ on $[0,\infty)$,   $\varfun{\sqrt{d_k}\,}$ is continuous on these intervals.  Now, addition and multiplication preserves continuity. Furthermore, the intersection of two (and, consequently, three) sets that are unions of finitely many intervals is again the union of finitely many intervals. Therefore, see Definition~\ref{deffffun}\eqref{deffffund}, the lemma holds for $f=f_1+f_2\cdot \sqrt{d_k} \in F_k\bix $.
\end{proof}

The \emph{restriction} of a map $\phi$ to a set $A$ is denoted by $\restrict \phi A$. The \emph{identity map} $A\to A$,  defined by $a\mapsto a$ for all $a\in A$, is denoted by $\idmap A$.

\begin{lemma}\label{KeYlmm}
Let $F$ be a subfield of $\mathbb R$, and let $c\in \mathbb R$ be a transcendental element over $F$. Assume that
\begin{equation}
F(c)=K_0\subset K_1\subset \cdots \subset K_k
\label{kTwrnlM}
\end{equation}
is a tower of quadratic field extensions such that $K_k\subseteq \mathbb R$. Then there exists a tower \eqref{towerxf} of quadratic field extensions over $F(x)$ of length $k$ and there are field isomorphisms $\phi_j\colon F_j\bix \to K_j$ for $j\in\set{0,\dots k}$ such that 
\begin{enumeratei}
\item\label{KeYlmma} $\restrict{\phi_0}F=\idmap F$, $\phi_0(x)=c$, and for all $j\in\set{1,\dots k}$,   $\restrict{\phi_j}{F_{j-1}\bix }=\phi_{j-1}$; 
\item\label{KeYlmmb} for all $j\in\set{0,\dots k}$ and $f\in F_j\bix $, we have $c\in\dom(\fun f)$ and $\phi_j(f)=\fun f(c)$.
\end{enumeratei}
\end{lemma}

\begin{proof} We prove the lemma by induction on $k$. By \eqref{uniquesmpltrans}, $\idmap F$ extends to an isomorphism $\phi_0\colon F(x)\to F(c)$ such that  $\phi_0(x)=c$  and $\restrict{\phi_0}F=\idmap F$. 
Thus, for every $g=\sum a_ix^i\in F[x]$, we have that
$\phi_0(g)=\sum \phi_0(a_i)\phi_0(x)^i = \sum a_i c^i = \ffun  g(c)$. 
Since $c$ is transcendental over $F$, $\fun f_2(c)\neq 0$ holds for all $f_2\in F[x]\setminus\set0$. Hence, we obtain from  Definition~\ref{deffffun}\eqref{deffffuna}--\eqref{deffffunb} that $c\in\dom(\fun f)$ and $\phi_0(f)=\fun f(c)$ hold for all $f\in F(x)=F_0(x)$. This settles the base, $k=0$, of the induction.

Next, assume that  $k>0$ and the lemma holds for $k-1$. By  \eqref{sfExt} and $K_k\subseteq \mathbb R$, there exists a positive real number $e_k\in K_{k-1}$ such that $K_k=K_{k-1}(\sqrt{e_k})$ and $e_k$ is not a square in $K_{k-1}$. Let $d_k=\phi_{k-1}^{-1}(e_k)$. Since $\phi_{k-1}\colon F_{k-1}\bix \to K_{k-1}$ is an isomorphism, $d_k$ is not a square in $F_{k-1}\bix$. Furthermore, by the induction hypothesis,
\begin{equation}
c\in \dom(\ffun{d_k})\quad\text{and}\quad 0< e_k = \phi_{k-1}(d_k) = \ffun {d_k}(c)\text. 
\label{ekdkhghTg}
\end{equation}
Define $F_k\bix $ with the help of this $d_k$. That is, $F_k\bix :=F_{k-1}\bix (\sqrt{d_k}\,)$. It follows from  \eqref{unqcnrp} that each element $f\in F_k\bix $ can be written uniquely in the canonical form $f=f_1+f_2\cdot \sqrt{d_k}$, where $f_1,f_2\in F_{k-1}\bix $. 
By \eqref{uniquesmplalg}, $\phi_{k-1}$ extends to a (unique) isomorphism  $\phi_k\colon F_k\bix \to K_k$,  and $\phi_k$ is defined by the rule
\begin{equation}
\phi_k(f_1+f_2\cdot \sqrt{d_k})= \phi_{k-1}(f_1) + \phi_{k-1}(f_2) \cdot \sqrt{e_k}\text.
\label{dkszTbnWc}
\end{equation}
Clearly, Part \eqref{KeYlmma} of the lemma holds. 
By the induction hypothesis, $c$ belongs to $\dom(\fun{f_1})\cap \dom(\fun{f_2})$. Combining this containment with \eqref{ekdkhghTg} and Definition \ref{deffffun}\eqref{deffffund}, we obtain that $c\in\dom (\fun f)$, for all $f=f_1+f_2\cdot \sqrt{d_k}\in F_k\bix $, written in canonical form. Hence, using \eqref{dkszTbnWc},   Definition \ref{deffffun}\eqref{deffffunc}--\eqref{deffffund}, 
\allowdisplaybreaks{
\begin{align*}
\phi_k(f)&= \phi_{k-1}(f_1) + \phi_{k-1}(f_2) \cdot \sqrt{e_k} = \fun{f_1}(c) + \fun{f_2}(c)\cdot  \sqrt{  \ffun {d_k}(c) }\cr
& = \fun{f_1}(c) + \fun{f_2}(c)\cdot  \varfun{ \sqrt{ d_k\,}} (c) =\fun f(c),
\end{align*}
}
as required. This completes the induction step and the proof of Lemma~\ref{KeYlmm}.
\end{proof}

\begin{proof}[Proof of Proposition~\ref{limprop}]
First, we only deal with the case  $d=0$.  
We can assume that $\lim_{j\to \infty}u_j\neq 0$, since otherwise $\lim_{j\to \infty}u_j\in F$ needs no proof. 
Consequently, for all but finitely many $j\in\mathbb N$, $u_j\neq 0$. After omitting finitely many initial members from our sequences, we can assume that, 
\begin{equation}
\text{for all $j\in\mathbb N$, $u_j\neq 0$.}
\label{ujnotzerus}
\end{equation}
Let $\algc{F(x)}$ denote 
the \emph{algebraic closure} of the field $F(x)$.  Note that no matter how we choose the $d_j$'s in \eqref{towerxf}, $F_k\bix \subseteq \algc{F(x)}$.
Consider one of our transcendental numbers, $c_j$, where  $j\in \mathbb N$. Since $u_j$ is a real quadratic number over $F(c_j)$, there exists a tower  \eqref{kTwrnlM}, with $c_j$ in place of $c$,  of quadratic field extensions such that $u_j\in K_k$. This tower and even its length $k$ may depend on $j$. By Lemma~\ref{KeYlmm}, we obtain a tower of quadratic field extensions \eqref{towerxf} together with an isomorphism $\phi_k^{(j)}\colon F_k\bix \to K_k$ and  an $f_j\in F_k\bix $, both depending on $j$,   such that 
\begin{equation}
\restrict {\phi_k^{(j)}}F=\idmap F,\text{ } \phi_k^{(j)}(x)=c_j, \text{ } c_j\in\dom(\fun{f_j}), \text{ and } u_j=  \phi_k^{(j)}(f_j) = \fun{f_j}(c_j)\text.
\label{YtthWTr}
\end{equation}
We can write $W$ in the unique form $W(x,y)=\sum_{\pair it\in B}a_{it}x^iy^t$, where $B$ is a finite non-empty subset of $\mathbb N_0\times \mathbb N_0$ and $a_{it}\in F$ for all $\pair it\in B$. We define $b_t:= \sum_{\set{i: \pair it\in B}} a_{it} x^i\in F[x]\subseteq F(x)\subseteq F_k\bix $ and $\ezg= \sum_{t\in \mathbb N_0} b_ty^t \in F_k\bix [y]$. This sum is  finite and  $\ezg(f_j)\in F_k\bix $.  
Actually, $g$ corresponds to $W$ under the canonical $F[x,y]\to F_k\bix[y]$ embedding.
Note that $\ezg$ does not depend on $j$.
(Note also that we cannot write  $\ffun\ezg$ here since $\ffun\ezg$ stands for a real-valued function.)  Observe that $g\neq 0$ in $F_k\bix[y]$, because $W\neq 0$.  Using     \eqref{YtthWTr}, we obtain that 
\begin{align*}
\phi_k^{(j)}\bigl( \ezg(f_j)\bigr) &= \phi_k^{(j)}\bigl(\sum_{t\in \mathbb N_0} b_t f_j^t\bigr) = \phi_k^{(j)}\Bigl(\sum_{t\in \mathbb N_0} \bigl( \sum_{\set{i: \pair it\in B}} a_{it} x^i \bigr)f_j^t\Bigr)\cr
&= \phi_k^{(j)}\bigl(\sum_{\pair it\in B}a_{it}x^i f_j^t\bigr) 
 = \sum_{\pair it\in B} \phi_k^{(j)}(a_{it})\, \bigl(\phi_k^{(j)}(x)\bigr)^i\, \phi_k^{(j)}(f_j)^t \cr
&=  \sum_{\pair it\in B} a_{it}\, c_j^i\, u_j^t =W(c_j,u_j)=0\text.
\end{align*}  
This implies that $g(f_j) =0$. Observe that $f_j\in F_k\bix \subseteq \algc{F(x)}$ and that $\ezg\in F_k\bix [y]\subseteq  \algc{F(x)}[y]$.  Since $g$, as a polynomial over $\algc{F(x)}$, has only finitely many roots in $\algc{F(x)}$,   $\set{f_j: j\in\mathbb N}$ is a finite subset of $\algc{F(x)}$. Therefore, after thinning the sequence $\tuple{c_j: j\in\mathbb N}$ if necessary, we can assume that $f_j$ does not depend on $j$. Thus, we can let $f:=f_j\in F_k\bix $. This allows us to assume that, from now on, $F_k\bix $ and the tower \eqref{towerxf} do not depend on $j$.    We obtain from \eqref{YtthWTr} that 
\begin{equation}
c_j\in\dom(\fun f)\text{ and } u_j=\fun f(c_j),\text{ for all }j\in\mathbb N\text.
\label{skdnDpj}
\end{equation}
We obtain from \eqref{ujnotzerus} and \eqref {skdnDpj} that  $f\neq 0$ in $F_k\bix $. Combining  Lemma~\ref{lMMdom} with    \eqref{skdnDpj},  $0=d<c_j$,   and $\lim_{j\to\infty}c_j=0$, we conclude that $(0,\epsilon)\subseteq \dom(\fun f)$ for some positive $\epsilon\in \mathbb R$. 
Therefore, Theorem~\ref{qdrthM}  applies and 
\begin{equation} 
\fun f(x) = \sum_{\ell =t}^\infty b_\ell\cdot x^{\ell \cdot 2^{-k}} =:\Sigma_1(x) \text{\quad in some strict right neighborhood of }0,
\label{bcksfCs}
\end{equation}
where $b_t\neq 0$,  $b_t, b_{t+1}, b_{t+2}\dots \in F$, and $t,b_t, b_{t+1}, b_{t+2}\dots $ are uniquely determined.  
We claim that $t$ in \eqref{bcksfCs} is non-negative. Suppose the contrary. Then the function 
\begin{equation}
f_1(z):= \sum_{\ell =t}^\infty b_\ell\cdot z^\ell
\label{f1fncTn}
\end{equation}
has a pole at 0.  It is straightforward to see  and we also know from  Spiegel et al \cite[page 175]{schaumcomplex} that $\lim_{z\to 0}|f_1(z)|=\infty$. Therefore, since $\lim_{j\to\infty} c_j^{2^{-k}}=0$ 
and since \eqref{bcksfCs} gives $\fun f(x) =f_1(x^{2^{-k}})$ in some strict right neighborhood of 0, we obtain that 
\[\lim_{j\to\infty } |u_j| = \lim_{j\to\infty } |\fun f(c_j)| = \lim_{j\to\infty }  |f_1(c_j^{2^{-k}})|=\infty\text.
\]
This contradicts the assumption $\lim_{j\to 0} u_j\in\mathbb R$. Thus, $t$ in \eqref{bcksfCs} is non-negative. Next, consider the power series $\Sigma_2(y):=\sum_{\ell=0}^\infty b_\ell \cdot y^\ell$ where, for $\ell<t$, we let $b_\ell:=0$. Since $\Sigma_2(y)= \Sigma_1(y^{2^k})$, $\Sigma_2(y)$ converges in some strict right neighborhood of 0. Therefore,  the radius of convergence of $\Sigma_2(y)$ is positive and $\Sigma_2(y)$ is continuous at   0; see, for example,  Wrede and Spiegel~\cite[page 272]{schaumreal}.
 Hence, using \eqref{skdnDpj}, \eqref{bcksfCs}, \eqref{f1fncTn}, and $\lim_{j\to \infty}c_j^{2^{-k}}=0$,
\begin{align*}
\lim_{j\to \infty}u_j= \lim_{j\to \infty}\fun f(c_j) = \lim_{j\to \infty}\Sigma_1(c_j)=  
\lim_{j\to \infty}\Sigma_2(c_j^{2^{-k}}) 
= \Sigma_2(0)=b_0\in F.
\end{align*}
This proves Proposition~\ref{limprop} in the particular case $d=0$.

Finally, if $d \in F$ is not necessarily 0, then we let $d':=0$,   $c_j':=c_j-d$,   $u_j':=u_j$, and $W'(x,y):=W(x+d,y)$. It is straightforward to see that the primed objects satisfy the conditions of Proposition~\ref{limprop}. Hence, applying the particular case, we conclude that $\lim_{j\to \infty}u_j=\lim_{j\to \infty}u'_j\in F$.
\end{proof}

\begin{proof}[Proof of Theorem~\ref{limthm}] With the assumptions of the theorem, let $F$ be the real quadratic closure of the field $\mathbb Q(a_1,\dots, a_m)$.  Since $d$ and the coefficients in $W$ are constructible from $a_1,\dots,a_m$, \eqref{bdWThrT} implies that  $d$ and these coefficients belong to $F$. So, $W$ is  polynomial over $F$ of two indeterminates. 
Since $F$ is a  countable field, so is  its algebraic closure, $\algc F$. Thus, the set of transcendental numbers over $F$ is everywhere dense in $\mathbb R$. Hence, we can pick a sequence $\tuple{c_j:j\in\mathbb N}$ of transcendental numbers over $F$ such that $c_j\in (d,d+\epsilon)$ for all $j\in\mathbb N$ and 
$\lim_{j\to\infty} c_j=d$. Define $u_j$ by $u_j:=u(c_j)$; it is constructible from $c_j,a_1,\dots,a_m$ by the assumptions. Using \eqref{bdWThrT}, we obtain that 
$u_j\in  \qclo {\mathbb Q(a_1,\dots, a_m,c_j)}= \qclo{ \bigl(\qclo {\mathbb Q(a_1,\dots, a_n)}(c_j)\bigr)}   =  \qclo{F(c_j) } $. 
That is, $u_j$ is a real quadratic number over $F(c_j)$, for all $j\in\mathbb N$. We have $W(c_j,u_j)=W(c_j,u(c_j))=0$. Condition \eqref{limthmd} of Theorem~\ref{limthm} yields that $\lim_{j\to\infty}u_j= \lim_{j\to\infty}u(c_j)$ exists and equals $\lim_{x\to d+0}u(x)\in\mathbb R$.
Now that all of its conditions are fulfilled, we can apply Proposition~\ref{limprop}. In this way, we obtain that   $\lim_{x\to d+0}u(x) =  \lim_{j\to\infty}u_j \in F$. Therefore,  \eqref{bdWThrT}  yields that 
$\lim_{x\to d+0}u(x)$ is constructible from $a_1,\dots, a_m$.  This completes the proof of  Theorem~\ref{limthm}. 
\end{proof}

\section{The Limit Theorem at work}\label{limatworksection}
The aim of this section is to prove the following lemma. 

\begin{lemma}\label{realtwolemma} If $n\geq 8$, then there exists a  \emph{transcendental} number $0<c \in \mathbb R$ such that the cyclic $n$-gon 
\begin{equation}
P_{n}(c)=P_{n}(\underbrace{1, \dots, 1}_{\textup{$\ell $ copies}},\underbrace{c, \dots, c}_{\textup{$n-\ell$ copies}}),
\text{ where }\,\, \ell =
\begin{cases}
7,&\text{if }n\neq 14,\cr
9,&\text{if }n = 14,\cr
\end{cases}
\label{dlZrGtn}
\end{equation}
exists but it is not constructible from $1$ and $c$.
\end{lemma}

\begin{proof}[Proof of Lemma~\ref{realtwolemma}]  Let  $n\geq 8$ .  It follows from \eqref{shcrnqltS} that $P_n(b)$,  defined as the polygon in \eqref{dlZrGtn} with $b$ instead of $c$,   exists for every real number $b\in (0,1)$. 
Suppose, for a contradiction, that  $P_n(c)$  is constructible for all transcendental numbers $c\in(0,1)$. Define a function $u\colon (0,1)\to \mathbb R$ by $u(x)=1/(2r_n(x))$, where $r_n(x)$ denotes the radius of the circumscribed circle of $P_n(x)$; see \eqref{dlZrGtn}. We are going to use Theorem~\ref{limthm}. Condition  \eqref{limthma} of  
Theorem~\ref{limthm}, denoted by \ref{limthm}\eqref{limthma},  clearly holds with   $\tuple{0,1,2,0,1}$ playing the role of $\tuple{d,\epsilon,m,a_1,\dots,a_m}$. Since $P_n(c)$ is constructible  from  $0$,  $1$, and $c$ for all $c\in(0,1)$, so is $u(c)$. Thus,  \ref{limthm}\eqref{limthmc} is satisfied. 
The polynomial  $\Wnone \ell{m-\ell}(1,x,y)   \in\mathbb Z[x,y]\setminus\set 0$ from Lemma~\ref{LmMrZT}\eqref{LmMrZTc} shows that   \ref{limthm}\eqref{limthmb} also holds. 
Since the geometric dependence of $r_n(x)$ on $x$ is continuous, $r_n(x)$ tends to the radius $r_\ell$ of the regular $\ell$-gon with side length 1  as  $x\to 0+0$. Hence, $\lim_{x\to 0+0}u(x)=1/(2r_\ell)\in \mathbb R$, and \ref{limthm}\eqref{limthmd} holds. Thus,  we obtain from Theorem~\ref{limthm} that $1/(2r_\ell)$ is constructible from $0$ and $1$. Therefore, so is the regular $\ell$-gon, which  contradicts the Gauss--Wantzel theorem; see \cite{wantzel}.
\end{proof}

\begin{remark}\label{schRfnl} 
The proof of Lemma~\ref{realtwolemma} above shows how one can complete Schreiber's original argument. As it is pointed out in Remark~\ref{whGnCmeans}, first we need a definition. Rather then the opposite of \eqref{pzCrVy}, 
``in general not constructible'' in \eqref{schrng5} has to be understood as  ``there are concrete side lengths such that the cyclic $n$-gon exists but not constructible''.  Second,  one 
would need a generalization of Lemma~\ref{LmMrZT}\eqref{LmMrZTc}  for the case where the side lengths can be pairwise distinct.  
In this way, one could use Theorem~\ref{limthm}  to prove \eqref{schrng5}. 
Since we are proving a stronger statement, Theorem~\ref{thmour}, we do not elaborate these details.
\end{remark}

\begin{remark} The constructibility of the regular pentagon was already known by the classical Greek mathematics.  Now, combining  Theorem~\ref{limthm} with Theorem~\ref{thmour}\eqref{thmourbb} and Lemma~\ref{LmMrZT}\eqref{LmMrZTc}, we obtain  an unusual and quite complicated way to prove this well-known fact, because the radius of $P(1,1,1,1,1)$ is the 
limit of the radius of $P(1,1,1,1,1,y)$ as $y\to 0+0$.
\end{remark}

\section{Turning transcendental to rational}
\label{hilbertsection}
The aim of this section is to prove the following theorem.

\begin{theorem}[Rational Parameter Theorem for geometric constructibility] 
\label{ratiothm}
Let $I\subseteq \mathbb R$ be a nonempty open interval, let $a_1,\dots, a_m$ be arbitrary real numbers, and  let $u\colon I\to \mathbb R$ be a continuous function. Denote $\mathbb Q(a_1,\dots,a_m)$ by $K$, and assume that there exists a polynomial $W\in K[x,y]\setminus\set 0$ such that $W(x,u(x))=0$ holds for all $x\in I$. 
If  there exists a number $c\in I$  such  that $c$  is transcendental over $K$
and   $u(c)$ is geometrically non-constructible from   $a_1,\dots, a_m$ and $c$, then there also exists a \emph{rational} number $c'\in I$ such that   $u(c')$ is non-constructible from   $a_1,\dots, a_m$ and~$c'$.
\end{theorem}

In many cases, the continuity of $u$ follows from geometric reasons. Our proof needs that $u$ is continuous.
 Since the continuity of $u$ does not follow if we define it implicitly by the polynomial equation $W(x,u(x))=0$, 
we  stipulate it separately in the theorem above.
Although Hilbert's irreducibility theorem, see later, is closely connected with  Galois groups,  we need the following lemma; this lemma will allow us to \emph{simultaneously} deal with the irreducibility of a polynomial and the order of the corresponding splitting field extension, which is the same as the order of the corresponding Galois group.  
The following lemma is an easy reformulation of Proposition~\ref{szrk}\eqref{szrkc}. Note that $h$ and the other polynomials in \eqref{dzHhWpva} and \eqref{dzHhWpvb} below are not uniquely determined in general.

\begin{lemma}\label{dckTg}
 For a subfield $L$ of $\mathbb R$, let $g(y)=b_ky^k+b_{k-1}y^{k-1}+\cdots + b_0 \in L[y]$ be an \emph{irreducible}  polynomial of degree $k\in\mathbb N$. Then the following two assertions hold.
\begin{enumeratei}
\item\label{dckTga} There are polynomials $g_0\in L[y_1,\dots,y_k]$, $g_1,\dots, g_k\in L[y]$, 
$h_1\in L[y,x]$, $h_2\in L[y]$, 
and an \emph{irreducible} polynomial $h\in L[y]$ such that 
\begin{align}
g(x) - b_k\cdot \bigl(x-g_1(y)\bigr)\cdots \bigl(x-g_k(y)\bigr)  &=  h(y)h_1(y,x) \,\, \text{ holds in } \,\, L[y,x]\label{dzHhWpva}\\
\text{and }\,\, g_0\bigl( g_1(y),\dots, g_k(y)\bigr)-  y  &= h(y)h_2(y) \,\, \text{ holds in } \,\, L[y]\text{.}\label{dzHhWpvb}
\end{align}
\item\label{dckTgb} No matter how $h, h_1, h_2, g_0, g_1,\dots, g_k$ are chosen in part \eqref{dckTga}, a root of $g$ is constructible over $L$ iff all roots of $g$ are constructible over $L$ iff the degree of $h$ is a power of $\,2$.
\end{enumeratei}
\end{lemma}

\begin{proof} First, we prove part \eqref{dckTga}.
Let $F$ denote the  splitting field of $g$ over $L$.  Then  there are $\alpha_1,\dots, \alpha_k\in F$, the roots of $g$, such that 
\begin{equation}
g(x)=b_k\cdot (x-\alpha_1)\dots(x-\alpha_k)\,\text{ in }F[x] \text{ and }F=L(\alpha_1,\dots,\alpha_k)\text.
\label{dlsZTZnggbfbW}
\end{equation}
Since the degree $[F:L]$ of the field extension is finite and $L$ is of characteristic 0, $F$ is a simple algebraic extension of $L$; see, for example, Dummit and Foote~\cite[Theorem 14.25 in page 595]{dummitfoote}.
Thus, there exists a $\beta\in F$ such that $F=L(\beta)$. Let $h\in L[y]$ be the minimal polynomial of $\beta$. Clearly, $h(y)$ is an irreducible polynomial. 
Using \eqref{dlsZTZnggbfbW} and $F=L(\beta)$,  we obtain polynomials $g_i$ over $L$ such that $\beta=g_0(\alpha_1,\dots,\alpha_k)$ and $\alpha_j=g_j(\beta)$ for $j\in\set{1,\dots,k}$. By the definition of these polynomials and  \eqref{dlsZTZnggbfbW}, 
\begin{align}
g_0(g_1(\beta),\dots, g_k(\beta)) -\beta&=0\,\text{ in } F\,\text{ and} \label{dlsrPtRwBYxa}
\\
g(x)-b_k\cdot (x-g_1(\beta))\dots(x-g_k(\beta))&=0\,\text{ in }F[x]\text.
\label{dlsrPtRwBYxb}
\end{align}
We know that  $F=L(\beta)\cong L[y]/(h(y))$ in the usual way, where $\beta$ corresponds to $y+\bigl(h(y)\bigr)$; see, for example, Dummit and Foote~\cite[pages 512--513]{dummitfoote}. Hence,   there is a unique surjective ring homomorphism $\phi\colon L[y]\to F$, defined by the rule $\phi(u(y))\mapsto u(\beta)$. The kernel of this homomorphism is  
\begin{equation}
\Ker(\phi)=\bigl(h(y)\bigr),\,\text{where  the principal ideal is understood in }\,L[y]\text.
\label{kerphiskdB}
\end{equation}
 Note that the restriction $\restrict \phi L$ of $\phi$ to $L$ is the identity map $\idmap L$. We can extend $\phi\cup\set{\pair xx}$ to a unique surjective ring homomorphism $\psi\colon L[y,x]=L[y][x]\to F[x]$. We claim that 
\begin{equation}  
\Ker(\psi)=\bigl(h(y)\bigr),\,\text{  understood in }\,L[y,x]\text.
\label{slHkjbmXq}
\end{equation}
To verify this, let  $u(y,x)=\sum_{i=0}^s u_i(y)x^i\in L[y,x]$ and compute:
\begin{equation*}
\begin{aligned}
\psi\bigl(u(y,x)\bigr)&=
\psi\bigl(\sum_{i=0}^s u_i(y)x^i\bigr) = \sum_{i=0}^s \phi\bigl(u_i(y)\bigr) x^i = 0\,\,\text{ in }\, F[x] \cr
&\iff (\forall i)\,  \phi\bigl(u_i(y) \bigr) =0   \overset{\eqref{kerphiskdB}}\iff (\forall i)\,  h(y)\mid u_i(y) \, \text{ in }\, L[y]\cr
&\iff h(y)\mid u(y,x)  \, \text{ in }\, L[y,x],\, \text{ as required.}
\end{aligned}
\end{equation*}
Using that   $\restrict\psi L=\restrict\phi L=\idmap L$, we obtain that $\psi\bigl(g(x)\bigr) = g(x)$. Since $\psi$ is a ring homomorphism and $\psi\bigl(x-g_i(y))=\psi(x)-\psi\bigl(g_i(y)\bigr)=x - \phi\bigl(g_i(y)\bigr) = x- g_i(\beta)$, 
\begin{align*}
\psi\bigl(g(x) &- b_k \cdot (x-g_1(y))\dots(x-g_k(y))  \bigr)
\cr
&=g(x)- b_k\cdot (x-g_1(\beta))\dots(x-g_k(\beta))
\overset{\eqref{dlsrPtRwBYxb}}= 0 \text .
\end{align*}
Combining this with \eqref{slHkjbmXq}, we obtain a polynomial $h_1(y,x)\in L[y,x]$ such that \eqref{dzHhWpva} holds. Similarly,  the definition of $\phi$ gives that
\[
\phi\bigl( g_0 ( g_1(y),\dots, g_k(y) ) -y\bigr)=
 g_0 ( g_1(\beta),\dots, g_k(\beta) ) -\beta
\overset{\eqref{dlsrPtRwBYxa}}= 0\text.
\]
Hence,  \eqref{kerphiskdB} yields a polynomial $h_2\in L[y]$ such that   \eqref{dzHhWpvb} holds.   
This proves part \eqref{dckTga} of Lemma~\ref{dckTg}.

Second,   in order to prove part \eqref{dckTgb},  assume that we have polynomials satisfying  the conditions in part \eqref{dckTga}, including  \eqref{dzHhWpva}  and \eqref{dzHhWpvb}. 
Since $h$ is irreducible, $F:=L[y]/ \bigl(h(y))\bigr)$ is a field.  Let $\beta:=y+\bigl(h(y)\bigr)$, which generates $L[y]/ \bigl(h(y)\bigr)$, and let $\alpha_i=g_i(\beta)=g_i(y) + \bigl(h(y)\bigr) \in  L[y]/ \bigl(h(y)\bigr)$, for $i\in\set{1,\dots,k}$. Since $h(\beta)=0$ and $h(y)$ is irreducible, it follows that $h$ is the minimal polynomial of $\beta$. 
Substituting $\beta$ for $y$ and using that $h(\beta)=0$, 
\eqref{dzHhWpva} shows that $g(x)=b_k\cdot  (x-\alpha_1)\cdots(x-\alpha_k)$ in $F[x]$. 
Hence, $F$  includes the splitting field $L(\alpha_1,\dots,\alpha_k)$  of $g$ over $L$. On the other hand, the same substitution into \eqref{dzHhWpvb} shows that $\beta=g_0(\alpha_1,\dots,\alpha_k)$ belongs to the splitting field. Hence,  $F=L(\beta)=  L(\alpha_1,\dots,\alpha_k)$ is (isomorphic to) the splitting field  of $g$. It is well-known that   $[F:L]=[L(\beta):L] =\deg h$; see, for example, Dummit and Foote~\cite[Theorem 13.4 in page 513]{dummitfoote}. Thus, 
part \eqref{dckTgb}  follows from  Proposition~\ref{szrk}\eqref{szrkc}. 
\end{proof}

Next, we recall   Hilbert's irreducibility theorem, \cite{hilbert}, from  
 Fried and Jarden~\cite[page 219 and Proposition 13.4.1 in page 242]{friedjarden}, or  \cite[Theorem 13.3.5 in page 241]{friedjarden} in a particular form we need it later; see also \cite{hilbertwiki} for a short introduction.

\begin{proposition}[Hilbert's irreducibility theorem]\label{hilprop}
 Let $K\subseteq \mathbb R$ be a field, and let $T=\tuple{T_1,\dots,T_s}$ and $Y=\tuple{Y_1,\dots, Y_k}$ be two systems of variables. Let $f_1(T,Y)$, \dots, $f_m(T,Y)$ be irreducible polynomials in $Y$ with coefficients in $K(T)$. That is, $f_1,\dots,f_m$ are irreducible in $K(T)[Y]$. Then there exists a system $a=\tuple{a_1,\dots, a_s}\in \mathbb Z^s$ of integers such that $f_i(a,Y)$ is defined and it is irreducible in $K[Y]$ for all $i\in\set{1,\dots, m}$.
\end{proposition}
The tuple $a=\tuple{a_1,\dots, a_s}$ above is called a common \emph{specialization} of the polynomials $f_i(T,Y)$. The transition from $f_i(T,Y)$ to $f(a,T)$ is called a \emph{substitution}, or the \emph{$T:=a$ substitution}. 
The statement above says that finitely many irreducible polynomials have a common integer specialization. This easily implies that
\begin{equation}
\parbox{11 cm} {The $f_i(T,Y)$ above have infinitely many common integer specializations.}
\label{infhilbert}
\end{equation}
To see this, suppose, for a contradiction, that there are only are finitely many common integer specializations, and all of them are in the list $a', a'', \dots$ .
Let $f':= Y_1^2- a'_1T_1$; then $f'(a',Y)$ is reducible. Similarly, if $f'':= Y_1^2- a''_1T_1$, then $f''(a'',Y)$ is reducible, etc. . Hence, applying Proposition~\ref{hilprop} to $f_1,\dots,f_m, f',f'',\dots$, we obtain a new common integer specialization of $f_1,\dots, f_m$. This is a contradiction, which proves \eqref{infhilbert}.

We will also need the following trivial fact, which says that, even if they are  not defined everywhere, substitutions  are partially defined ring homomorphism:
\begin{equation}
\parbox{10 cm}
{For $f(T,Y), g(T,Y)\in K(T)[Y]$ and a specialization  $a\in \mathbb Z^s$, if the $T:=a$ substitution is defined for  $f(T,Y)$ and $g(T,Y)$, then it is also defined for and commutes with their sum and product.
}
\label{prthMrpH}
\end{equation}

\begin{proof}[Proof of Theorem~\ref{ratiothm}] Since $K[x,y]=K[x][y]\subseteq K(x)[y]$, we have that $W\in K(x)[y]$. 
We can assume that $W$ is  an irreducible polynomial  in $ K(x)[y]$. 
Suppose that this is not the case. Then there are a finite index set $J$,  pairwise non-associated irreducible polynomials  $W_j\in  K(x)[y]$, and $\alpha_j\in \mathbb N$ such that  $W(x,y) = \prod_{j\in J} W_j(x,y)^{\alpha_j}$. Using the isomorphism $K(x)\cong K(c)$, see \eqref{uniquesmpltrans}, we obtain that the $W_j(c,y)\in K(c)[y]$ are also irreducible.
From $W(c,u(c))=0$, we obtain that $u(c)$ is a root of some $W_j(c,y)$. Since non-associated irreducible polynomials cannot have a common root, there is exactly one $j\in J$ such that $W_j(c,u(c))=0$. Let $M:=\min\bigl\{|W_i(c,u(c))|: i\in J\setminus\set j  \bigr\}$. By the uniqueness of $j$, $M>0$. 
The roots of the $W_i(c,y)$ depend continuously on the parameter $c$, and  $u$ is a continuous function. Hence,  there is a small neighborhood of $c$ such that for all $c'$ in this neighborhood and for all $i\in J\setminus\set j$, we have that $|W_i(c',u(c'))|>M/2$. However, $0=W(c',u(c'))=\prod_{i\in J} W_i(c',u(c'))^{\alpha_i}$ for all  $c'\in I$. Thus, $W_j(c',u(c'))=0$ for all $c'$ in a small neighborhood of $c$. Therefore, after replacing $I$ by this small neighborhood,  $W$ can be replaced by $W_j$   in our considerations; this justifies the assumption that  $W$ is irreducible in $ K(x)[y]$. Also, $W(c,y)$ is irreducible in $K(c)[y]$.  Actually, we can assume even more. If we multiply (or divide) it by an appropriate polynomial from $K[x]$ if necessary,  $W(x,y)$ becomes an irreducible polynomial in $K[x][y]$ by  Gauss' Lemma; see Dummit and Foote~\cite[Proposition 9.5 and Corollary 9.6 in pages 303--304]{dummitfoote}. So, we assume  $W(x,y)$  
is \emph{irreducible}  in $K[x][y]=K[x,y]$  and also in $K(x)[y]$. 
We  write $W(x,y)$ in the form 
\begin{equation}
W(x,y)=a_k(x)y^k + a_{k-1}(x)y^{k-1}+\cdots+a_0(x), \text{ where } a_k,\dots, a_0\in  K[x],
\label{sdzWTgnb}  
\end{equation}
$a_k(x)\neq 0$ and, since $W$ is irreducible in $K[x][y]$, $a_0(x)\neq 0$. Note that $a_k(c)\neq 0$ and $a_0(c)\neq 0$, since $c$ is transcendental over $K$.  
After shrinking the interval $I$ if necessary, we can assume that 
\begin{equation}
\text{for all }r\in I,\text{ }a_k(r)\neq 0 \text{ and } a_0(r)\neq 0\text.
\end{equation}

We are looking for an appropriate rational number $c'$ within $I$, that is, near $c$. However, \eqref{infhilbert} can only give some very large $c'\in \mathbb Z$, which need not belong to $I$. To remedy this problem, we are going to translate the constructibility problem of $u(x)$ to an equivalent problem that is easier to deal with. 
To do so, we can assume that the open interval $I$ is determined by two rational numbers, because otherwise we can take a smaller interval that still contains $c$. Let $q$ be the middle point of $I$; it is a rational number, and $I$ is of the form $I=(q-\delta,q+\delta)$, where $0<\delta\in\mathbb Q$. 
We can assume that $q=0$, because otherwise we can work with 
\[\tuple{u^\ast(x):=u(x+q),\, W^\ast(x,y):=W(x+q,y),\, c^\ast:=c-q }
\] 
instead of $\tuple{ u(x), W(x,y), c}$; to justify this, observe that $c^\ast$ is still transcendental over $K$. Note also that, by   the uniqueness of simple transcendental field extensions,  see \eqref{uniquesmpltrans}, $\id_K \cup \set{\pair x{x+q}}$ extends first to an automorphism of $K(x)$, and then to an automorphism of $K(x)[y]$ that maps $W$ to $W^\ast$, and we conclude that   $W^\ast$ is irreducible over $K(x)$. We can also assume that $W^\ast$ is irreducible over $K[x]$, because otherwise we can divide it by the greatest common divisor of its coefficients, which belongs to $K[x]\setminus\set 0$, and Gauss' Lemma applies.
Finally, for every $x\in\mathbb R$,  $x+q$ is constructible from $x$ and vice versa; so they are equivalent data modulo geometric constructibility.

So, from now on, $I=(-\delta, \delta)$. 
The punctured interval $I \setminus\set 0 =(-\delta,0)\cup(0,\delta)$ will 
be denoted by $\pnct I$. We also consider another open set,  
 \[J:= \set{r\in\mathbb R: |r|>1/\delta} =( -\infty,-1/\delta) \cup (1/\delta,\infty)
\text.
\] 
The functions $\tau\colon \pnct I\to J$, defined by $\tau(x):= 1/x$, and $\xi: J\to\pnct I$, defined by $\xi(t):= 1/t$, are reciprocal bijections and both are continuous on their domains. To emphasize that $\xi$ and $\tau$ have disjoint domains, we often write  $\xi(t)$ and $\tau(x)$ instead of $1/t$ and $1/x$, respectively. 
The compound function, $v\colon J\to \mathbb R$, defined by $v(t)=u(\xi(t))$ is also continuous. 
 Since $\tau(x)$ is constructible from $x$ and $\xi(t)$ is constructible from $t$, we obtain that
\begin{equation}
\parbox{9.5 cm}
{The constructibility of $u(x)$ from $x\in \pnct I$ over $K$ is equivalent to the constructibility of $v(t)= u(\xi(t))$ from $t\in J$ over $K$.}
\label{consteQu}
\end{equation} 
Of course, ``from $x\in\pnct I$ over $K$'' is equivalent to ``from $x\in\pnct I$ and $a_1,\dots, a_m$'', and a similar comment applies for analogous situations. 
Using that  both $\xi$ and $\tau$ map rational numbers to rational numbers, we conclude from \eqref{consteQu} that 
\begin{equation}
\parbox{6.5cm}{It suffices to find a $d'\in J\cap\mathbb Q$ such that $v(d')$ is not constructible from $a_1,\dots, a_m$.} 
\label{dhgmPcTg}
\end{equation}
Let $d=\tau(c)=1/c$; it is also transcendental over $K$. We obtain from  \eqref{consteQu} that  $v(d)$ is not constructible from $d$ over $K$. 
For $i\in \set{0,\dots, k}$, we define  $\wt b_i(t):=a_i(\xi(t))=a_i(1/t)\in K(t)$. Clearly, for a sufficiently large $j\in\mathbb N$, we obtain that  $t^j\cdot\wt b_0(t)$,  \dots, $t^j\cdot \wt b_0(t) \in K[t]$, and we can take out the greatest common divisor from these polynomials. Hence, there is a polynomial $q(t)\in K[t]$ such that  $\wh b_0(t):=q(t)\wt b_0(t)\in K[t]$, \dots, $\wh b_k(t):=q(t)\wt b_k(t)\in K[t]$, and the greatest common divisor of  $\wh b_0(t)$, \dots,  $\wh b_k(t)$    is $1$. Observe that $1/t\in K(t)$ is a transcendental element over $K$ and $K(t)=K(1/t)$. Thus, by the uniqueness of simple transcendental field extensions, see \eqref{uniquesmpltrans}, $\id_K\cup\set{\pair {x}{1/t}}$ extends to an isomorphism of $K(x)\to K(t)$, and  also to an isomorphism $K(x)[y]\to K(t)[y]$ such that $y\mapsto y$. 
This isomorphism maps $a_i(x)$ to $a_i(1/t)=\wt b_i(t)$. Using this isomorphism and the fact that  \eqref{sdzWTgnb} is an irreducible polynomial over $K(x)$, we obtain that 
 $\wt b_k(t)y^k+\cdots \wt b_0(t) \in K(t)[y]$ is  irreducible over $K(t)$. Multiplying this polynomial by $q(t)$, we obtain that
\begin{equation}
\wh W(t,y):= \wh b_k(t)y^k +\wh b_{k-1}(t)y^{k-1}+\cdots +\wh b_0(t) \in K[t][y]
\label{whWty}
\end{equation}
is a polynomial that is irreducible in $y$ over  $K[t]$ and also over $K(t)$. 
For $t\in J$, let us compute:
\begin{equation}
\begin{aligned}
\wh W(t,v(t)) &= \wh b_k(t)\cdot v(t)^k + \wh b_{k-1}(t) \cdot v(t)^{k-1}  +\cdots + \wh b_0(t)\cr
 &= q(t)\bigl(  a_k(\xi(t))\cdot  u(\xi(t))^k+  \cdots+  a_0(\xi(t)) \bigr)  \cr
&=q(t)\cdot W\bigl( \xi(t), u(\xi(t))\bigr) = 0,
\end{aligned}
\label{dhJdqGpT}
\end{equation}
because $\xi(t)\in \pnct I$.
Therefore, $v(t)$, whose constructibility from $t$ over $K$ is investigated, is the root of the irreducible polynomial $\wh W(t,y)\in K[t][y]$ in the sense that $\wh W(t,v(t))=0$ for all $t\in J$, and we know that $v(d)$ is not constructible from the transcendental number $d\in J$ over $K$.

Let $L=K(d)$. Based on \eqref{whWty}, we let 
\[\jelz gd(y)=\wh b_k(d)y^k +\wh b_{k-1}(d)y^{k-1}+\cdots+ \wh b_0(d) \in L[y]\text.
\] 
Note that $\jelz gd(y)$ is obtained from  $\wh W(t,y)$ by the $t:=d$ substitution.
Since the polynomial in \eqref{whWty} is irreducible in $y$ over $K(t)$ and $L=K(d)\cong K(t)$, we conclude that $\jelz gd$ is irreducible   in $L[y]$. Since $d$ is transcendental over $K$, we can substitute $d$ for $t$ in \eqref{dhJdqGpT}. Note that $\wh W$ is composed from polynomials and  $\xi(t)=1/t$, so 
\begin{equation}
\text{any substitution for $t$ in \eqref{dhJdqGpT} makes sense except for $t:=0$.}
\label{anysuBst}
\end{equation}
In this way, we obtain that $v(d)$ is a root of $\jelz gd$. By Lemma~\ref{dckTg}, there is an irreducible polynomial $\jelz hd(y)\in L[y]$ and there are further  polynomials $\jelz gd_0\in L[y_1,\dots, y_k]$, $\jelz gd_1,\dots, \jelz gd_k, \jelz hd_2\in L[y]$, and  $\jelz hd_1\in L[y,x]$  such that   \eqref{dzHhWpva} and \eqref{dzHhWpvb}, with the superscript $(d)$ added, hold. Since $v(d)$ is not constructible over $L$, Lemma~\ref{dckTg} yields that  $\degx y{\jelz hd}$ is not a power of 2.

Using that $\id_K\cup\set{\pair dt}$ extends to a unique field isomorphism $\phi_0\colon L=K(d)\to K(t)$, which extends further to a unique ring isomorphism $\phi\colon L[y]\to K(t)[y]$ that maps $y$ to $y$. Let $\wh h(t, y)$ denote the $\phi$-image of $\jelz hd(y)$; we will also denote it by $\jelz ht(y)$.  
(Note that $\jelz hd(y)$   is obtained from  $\wh h(t,y)$ by the $t:=d$ substitution; however $d$ as a specialization is not so important, because we are only interested in \emph{integer} specializations.) 
Clearly, $\degx y{\jelz ht}=\degx y{\jelz hd}$, because $d$ is transcendental over $K$. Using that $\phi$ preserves irreducibility, we obtain that  $\jelz gt(y)$ and $\jelz ht(y)$ are irreducible polynomials in $ K(t)[y]$, since so are $\jelz gd(y)$ and $ \jelz hd(y)$ in  $L[y]$.
We define  $\jelz gt, \jelz gt_1,\dots,\jelz gt_k,\jelz ht_2\in K(t)[y]$  as the $\phi$-images of  $\jelz gd, \jelz gd_1,\dots, \jelz gd_k, \jelz hd_2\in L[y]$, respectively. 
Note at this point that    $\jelz gt(y)=\wh W(t,y)$; see \eqref{whWty}.  Hence, \eqref{dhJdqGpT} and \eqref{anysuBst} give that
\begin{equation}
\text{For every $d'\in J$, $v(d')$ is a root of $\jelz g{d'}(y):=\wh W(d',y)$.} 
\label{vrtttdPr}
\end{equation}
We also define $\jelz gt_0\in K(t)[y_1,\dots, y_k]$ as the image of $\jelz gd_0$ under the unique  extension of $\phi_0 \cup\set{\pair {y_1}{y_1},\dots,\pair {y_k}{y_k}}$ to a (unique) isomorphism $L[y_1,\dots, y_k]\to  K(t)[y_1,\dots, y_k]$. We define $\jelz ht_1\in K(t)[y,x]$ analogously.  In this way, we have defined a ``$(t)$-superscripted variant'' over $K(t)$  of the system of polynomials occurring in Lemma~\ref{dckTg}. Since isomorphisms preserve \eqref{dzHhWpva} and \eqref{dzHhWpvb}, we conclude that 
\begin{equation}
\parbox{6cm}
{The $(t)$-superscripted family of our polynomials satisfy \eqref{dzHhWpva} and \eqref{dzHhWpvb}.}
\label{tSpCRhlds}
\end{equation}

Next, we apply  Proposition~\ref{hilprop} with $s=1$ for the irreducible polynomials $\jelz gt(y)=\wh W(t,y)$ and $\jelz ht(y)=\wh h(t,y)$. In this way,  taking \eqref{infhilbert} into account, we conclude that there are infinitely many specializations $d'\in \mathbb Z$ such that 
\begin{equation}
\parbox{7cm}
{$\jelz h{d'}(y):= \wh h(d',y)$ and $\jelz g{d'}(y)=\wh W(d',y)$ are defined and they are irreducible in $K[y]$.}
\label{dPzGrtjbnRqW}
\end{equation}   
We also specialize the polynomials $\jelz gt_0, \jelz gt_1,\dots,\jelz gt_k,\jelz ht_1,\jelz ht_2$  by $t:=d'$. Not every $d'$ satisfying \eqref{dPzGrtjbnRqW} is appropriate for this purpose, because the coefficients of these polynomials are fractions over $K[t]$ and some denominators may turn to zero if we substitute $d'$ for $t$. Fortunately, there are only finitely many denominators with finitely many roots, so we still have infinitely many $d'\in Z$ that specialize all these additional polynomials such that \eqref{dPzGrtjbnRqW} still holds. There are only finitely many of these $d'$ such that the leading coefficient of $\jelz ht$ diminishes  at the specializations $t:=d'$, and there are also finitely many $d'$ outside $J$. Thus, there exists a $d'\in J\cap \mathbb Z$ such that \eqref{dPzGrtjbnRqW} holds, all the ``$(t)$-superscripted'' polynomials can be specialized by $t:=d'$, and  $\degx y{\jelz h{d'}} = \degx y{\jelz h{t}}=\degx y{\jelz h{d}}$, which is not a power of $2$.    Combining  \eqref{prthMrpH}  and  \eqref{tSpCRhlds}, we obtain that the $(d')$-superscripted family of polynomials satisfy \eqref{dzHhWpva} and \eqref{dzHhWpvb}. Therefore, by Lemma~\ref{dckTg}\eqref{dckTgb}, no root of $\jelz g{d'}(y)\in K[y]$ is constructible over $K$. But $v(d')$ is a root of $\jelz g{d'}(y)$ by \eqref{vrtttdPr}, whence $v(d')$ is not constructible over $K$. That is, $v(d')$ is not constructible from $a_1,\dots, a_m$ that generate the field $K$. Thus,  \eqref{dhgmPcTg} completes the proof of Theorem~\ref{ratiothm}.
\end{proof}

\section{Completing the proof of Theorem~\ref{thmour}}\label{compLsect}

We only have to combine  some earlier statements.
\begin{proof}[Proof of Theorem~\ref{thmour}]
Parts \eqref{thmoura} and \eqref{thmourb} are proved  in Section~\ref{smallorder}. Part~\eqref{thmourc} is proved in Section~\ref{evensection} in an elementary way. 
For $n\geq 8$, Lemma~\ref{realtwolemma} proves that there exists an appropriate $\ell\in\set{7,9}$ and a positive transcendental number $c\in\mathbb R$ such that $P_n(c)=P(1, \dots,1,c,\dots,c)$ given in \eqref{dlZrGtn} exists but it is not constructible from its sides. In other words, $P_n(c)$ is geometrically not constructible  from $0$, $1$ and $c$.
 We conclude from \eqref{shcrnqltS}  that there exists a  small open interval $I$ such that $c\in I$ and $P_n(x)=P(1, \dots,1,x,\dots,x)$ exists for all $x\in I$.  Let $r_n(x)$ denote the radius of the circumscribed circle of $P_n(x)$.
Now, we are in the position to apply Theorem~\ref{ratiothm} with $m=2$, $a_1=0$, $a_2=1$, $u(x)=1/(2r_n(x))$,  and  $W(x,y):=\Wnone {\ell}{n-\ell}(1,x,y)$; see Lemma~\ref{LmMrZT}\eqref{LmMrZTc}. In this way, 
we can change  $c$ to an appropriate rational number $c'\in I$ such that $P_n(c')=P(1,\dots,1,c',\dots,c')$ still exists but it 
is not constructible from its sides, $1$ and $c'$. We can write $c'$ in the form $c_1'/c_2'$ where $c_1', c_2'\in \mathbb N$.  Clearly,  $P(c_2',\dots,c_2',c_1',\dots,c_1')$ is non-constructible, because it  is geometrically similar to $P_n(c')$. Therefore, $n\in\setncl 2$. This proves parts~\eqref{thmourc} and  \eqref{thmourd} of Theorem~\ref{thmour}. 
\end{proof}


\begin{thebibliography}{99}
\bibitem{newton}  
 Newton, Isaac:
 The method of fluxions and infinite series; with its application to the geometry of curve-lines. 
 Translated from Latin and published by John Colson in 1736.
\end{thebibliography}
\end{document}